\documentclass[12pt,a4paper,reqno]{amsart}
\usepackage{amssymb,bm}
\usepackage{latexsym}
\usepackage{exscale}
\usepackage{mathrsfs}
\usepackage{color}
\usepackage{graphicx}

\usepackage{hyperref}
\usepackage[left=2.5cm,right=2.5cm,top=2cm,bottom=2cm]{geometry}
\usepackage{amssymb}
\usepackage{amsrefs}
\usepackage{amsmath}

\numberwithin{equation}{section}
\newtheorem{thm}{Theorem}[section]

\newtheorem{cor}[thm]{Corollary}

\newtheorem{lemma}[thm]{Lemma}
\newtheorem{proposition}[thm]{Proposition}

\newtheorem{defi}[thm]{Definition}
\newtheorem{remark}[thm]{Remark}
\usepackage{bbm}
\newcommand{\cI}{\ensuremath{\mathcal{I}}}
\newcommand{\cP}{\ensuremath{\mathcal{P}}}

\newcommand{\FF}{\ensuremath{\mathbb{F}}}
\newcommand{\RR}{\ensuremath{\mathbb{R}}}
\newcommand{\NN}{\ensuremath{\mathbb{N}}}
\newcommand{\bn}{\ensuremath{\mathbf{n}}}
\newcommand{\bd}{\ensuremath{\mathbf{d}}}
\newcommand{\CC}{\ensuremath{\mathbb{C}}}
\newcommand{\xx}{\ensuremath{\bm{x}}}

\newcommand{\XX}{\ensuremath{\mathbb{X}}}
\newcommand{\XB}{\ensuremath{\mathcal{X}}}

\newcommand{\YB}{\ensuremath{\mathcal{Y}}}

\newcommand{\Ind}{\ensuremath{\mathbbm{1}}}

\newcommand{\EE}{\ensuremath{\mathbb{E}}}
\newcommand{\cO}{\ensuremath{\mathcal{O}}}

\newcommand{\Cu}{\ensuremath{\mathcal{Q}}}
\newcommand{\cS}{\ensuremath{\mathcal{S}}}

\newcommand{\cG}
{\ensuremath{\mathcal{G}}}
\newcommand{\cF}
{\ensuremath{\mathcal{F}}}

	\DeclareMathOperator{\sgn}{sign}
	
	\DeclareMathOperator{\supp}{supp}
	\DeclareMathOperator{\osc}{osc}

	\usepackage[shortlabels]{enumitem}

	\newcommand{\floor}[1]{\left\lfloor #1 \right\rfloor}
	\newcommand{\ceil}[1]{\left\lceil #1 \right\rceil}

\usepackage[title]{appendix}

\begin{document}
	
	\title[Summability methods for the greedy algorithm in Banach spaces]{Summability methods for the greedy algorithm in Banach spaces}%
	
	\author{M. Berasategui}
	\address{Miguel Berasategui, UBA-PAB I, Facultad de Ciencias Exactas y Naturales, Universidad de Buenos Aires, (1428), Buenos Aires, Argentina } 
	\email{mberasategui@dm.uba.ar}
	
	\author{P. M. Bern\'a}
	\address{Pablo M. Bern\'a, Departamento de Matemáticas, CUNEF Universidad, Madrid 28040, Spain} 
	\email{pablo.berna@cunef.edu}

		\author{S. J. Dilworth}
	\address{Stephen J. Dilworth, 	Department of Mathematics, University of South Carolina, Columbia, SC 29208, USA} 
	\email{dilworth@math.sc.edu}
	
	\author{D. Kutzarova}
	\address{Denka Kutzarova, Department of Mathematics, University of Illinois, Urbana IL 61801, USA} 
	\email{denka@illinois.edu}
	

	
	\thanks{The first author was supported by the Grants ANPCyT PICT 2018-04104 and CONICET PIP 11220200101609CO. The second author was supported by the Grant PID2022-142202NB-I00 (Agencia Estatal de Investigación, Spain). The third author was supported by Simons Foundation Collaboration Grant No. 849142.}
	
	\begin{abstract}
		{For the past 25 years, one of the most studied algorithms in the field of Nonlinear Approximation Theory has been the Thresholding Greedy Algorithm. In this paper, we propose new summability methods for this algorithm, generating two new types of greedy-like bases - namely  Cesàro quasi-greedy and de la Vallée-Poussin-quasi-greedy bases. We analyze the connection between these types of bases and the well-known quasi-greedy bases, and leave some open problems for future research.  In addition, as a consequence of our techniques for handling these summability methods, we answer a question posed by P. Wojtaszczyk in \cite{W}, by giving a categorial proof of equivalence between the uniform boundedness of the greedy sums and the convergence of the thresholding greedy algorithm.}
	\end{abstract}

	\maketitle
	
	\section{Introduction}
	
	For the past 25 years, an algorithm frequently studied in the field of Nonlinear Approximation Theory is the so-called \textit{Thresholding Greedy Algorithm} $(G_m)_{m\in\mathbb N}$ (TGA for short), where the basic idea is that, given an element of a space, the algorithm selects the largest coefficients of this element with respect to a given basis of the space. To be more precise, if $x\in\mathbb X$ with $\XX$ a Banach space and $\mathcal B=(\xx_n)_{n\in\mathbb N}$ is a basis, then 
$ G_m(f)=\sum_{n\in A_m(x)}\xx_n^*(x)\xx_n,$
where $A_m(x)$ is called a \textit{greedy set of $x$ with cardinality $m$} and verifies that 
$$\min_{i\in A_m(x)}\vert\xx_i^*(x)\vert\geq\max_{j\not\in A_m(x)}\vert\xx_j^*(x)\vert.$$

The researchers responsible for introducing this algorithm were V. N. Temlyakov and S. V. Konyagin in \cite{KT1999} in the year 1999. Since then, different researchers such as F. Albiac, J. L. Ansorena, S. J. Dilworth, N. J. Kalton, D. Kutzarova, T. Schlumprecht, P. Wojtaszczyk, etc., have shown interest in this algorithm and have produced several research papers analyzing different types of convergence and problems related to the TGA. 

One of the most important notions introduced in 1999 in \cite{KT1999} was the concept about quasi-greedy bases, where a basis is quasi-greedy when there is $C>0$ such that
$$\Vert G_m(x)\Vert\leq C\Vert x\Vert,\; \forall m\in\mathbb N, \forall x\in\mathbb X.$$
This notion has a connection with the convergence of the TGA and this was observed by P. Wojtaszczyk in \cite{Wo2000}, where he proved that a basis is quasi-greedy if and only if  
$$\lim_{m\rightarrow +\infty}\Vert x-G_m(x)\Vert=0,\; \forall x\in\mathbb X.$$

Thanks to this equivalence, it can be observed that the condition of being quasi-greedy is the weakest condition to ensure the convergence of the algorithm, since we know that the algorithm converges but nothing more; that is, we do not know how fast it converges or any other additional properties.

But what happens if the basis is not quasi-greedy? Is it possible to work with the TGA to guarantee a weaker form of convergence in a different sense? In Functional Analysis, \textit{summability methods} are techniques used to extend the concept of convergence for series and sequences, particularly when the series do not converge in the classical sense (such as pointwise or norm convergence). These methods are powerful tools for handling divergent series and assigning them a meaningful value for analysis. For instance, one classical method is the \textit{Cesàro summability}, which smooths out the partial sums of a series. Given a sequence of partial sums \( S_n \) of a series \( \sum a_n \), the Cesàro sum is defined by averaging the partial sums:
\[
\mathcal C_n := \frac{1}{n} \sum_{k=1}^{n} S_k.
\]
If the sequence $(\mathcal C_n)_n$ converges to a value \( L \), we say the series is Cesàro-summable to \( L \), even if the series does not converge in the usual sense.

In this paper, we focus our attention on the definition of Cesàro-type summability for the TGA and the study of the interaction between this new convergence with other classical types of convergence studied in the field of greedy approximation. The structure of the paper is as follows: in Section \ref{main}, we introduce the main definitions about greedy-like bases and the new ones based on Cesàro and de la Vallée-Poussin-summability: Cesàro-quasi-greedy and de la Vallé-Poussin-quasi-greedy bases. In Section \ref{characterization}, we characterize Cesàro-quasi-greedy and de la Vallée-Poussin-quasi-greedy bases following the spirit of the characterization of quasi-greediness as in \cite{Wo2000}. Section \ref{tech} is dedicated to prove some technical results used in the following sections. In Section \ref{subsequences}, we will study de la Vallée-Poussin-quasi-greedy bases through subsequences, following the generalization of greedy-like bases introduced by T. Oikhberg in \cite{O2018} and, finally, in Section \ref{relation}, we show some connections between Cesàro-quasi-greedy, de la Vallée-Poussin-quasi-greedy bases and some classical greedy-like bases. 

\section{Main definitions and Notation}\label{main}
As we have mentioned before, throughout this article, we assume that we are working with a Banach space $\mathbb{X}$ over $\FF$ ($\FF=\RR$ or $\FF=\CC$) endowed with a norm $\Vert \cdot \Vert$ and a  \textit{basis} $\XB=(\xx_n)_{n\in\mathbb N}$, that is, 
\begin{itemize}
	\item $\overline{\langle \xx_n : n\in\mathbb N\rangle}=\XX$;
	\item there exists a biorthogonal sequence $\XB^*=(\xx_n^*)_{n\in\NN}\subset \XX^*$.
\end{itemize}
$\XB^*$ is called the \textit{dual basis}. A basis is a \emph{Markushevich basis} if 
\begin{itemize}
	\item $	\xx_n^*\left(x\right)=0\;\forall n\in \NN\Longrightarrow x=0,$
\end{itemize}
and it is \emph{strong} (see \cite{AABW2021}) if, for every infinite subset  $A$ of $\NN$, 
\begin{align*}
	\left\lbrace x\in \XX: \xx_n^*\left(x\right)=0\;\forall n\in \NN\setminus A\right\rbrace=\overline{\langle \xx_n: n\in A  \rangle}.
\end{align*}

Working with bases, we can define the \textit{support} of $x\in\XX$ as the set
$$\text{supp}(x):=\lbrace n\in\mathbb N : \xx_n^*(x)\neq 0\rbrace.$$
Also, related to a basis $\XB$ in a Banach space $\XX$, we can define the following  quantities:  
\begin{align*}
	&\alpha_1=\alpha_1\left[\XB, \XX\right]:=\sup_{n\in \NN}\|\xx_n\|;\\
	&\alpha_2=\alpha_2\left[\XB, \XX\right]:=\sup_{n\in \NN}\|\xx_n^*\|;\\ &\alpha_3=\alpha_3\left[\XB, \XX\right]:=\sup_{n\in \NN}\|\xx_n\|\|\xx_n^*\|. 
\end{align*}
Unless otherwise stated, we assume that the $\alpha_j$ above are finite (equivalently, all of our bases and dual bases are seminormalized). 

	\subsection{Sets}
Given a set $A$, $\left\vert A\right\vert $,  denotes its cardinality, $\cP\left(A\right)$ denotes its power set; for each $n\in \NN$, $\cI_{n}$ denotes the set $\{1,\dots, n\}$.\\
Let $\XB=(\xx_n)_{n\in\mathbb N}$ be a basis of $\XX$, $x\in \XX$, and $m\in \NN_0$. By $\cG\left(x,m \right)$ we denote the set of greedy sets for $x$ of cardinality $m$, that is, sets $A\subset \NN$ with $\left\vert A\right\vert =m$ such that
\begin{align*}
	&\left\vert \xx_n^*\left(x\right)\right\vert\ge \left\vert \xx_k^*\left(x\right)\right\vert&&\forall n\in A\forall k\not\in A. 
\end{align*}

Given a Banach space $\XX$ over a field $\FF$, $\EE$ denotes the unit sphere of $\FF$, and we set $\kappa=1$ if $\FF=\RR$ and $\kappa=2$ if $\FF=\CC$. 

Given a basis $\XB$ for a Banach space $\XX$ with dual basis $\XB^*$, for each $x\in \XX$ we define $\varepsilon\left(x\right):=\left(\sgn\left(\xx_n^*\left(x\right)\right)\right)_{n\in \NN}$, where $\sgn(0)=1$ and $\sgn(a)=\frac{a}{|a|}$ for all $a\in \FF\setminus\{0\}$.

By $\NN^{<\infty}$ we denote the set of finite subsets of $\NN$. 
By $\NN^{(m)}$ we denote the set of subsets of $\NN$ of cardinality $m$, whereas by  $\NN^{\le m}$ we denote the set of subsets of $\NN$ of cardinality no greater than $m$. 
For each $k\in \NN$, $k\NN$ denotes the set of multiples of $k$. 
Given $A\subset \RR$ and $a\in \RR$, we note $A_{>a}:=\left\lbrace x\in A: x>a\right\rbrace$. We use similar notation for $\le$, $\ge$ and $<$.  
If $A, B\subset \RR$, we note $A<B$ if $a<b$ for all $a\in A$, $b\in B$. Also, by $m<B$ we mean $\{m\}<B$. We use similar notation for $\le$, $\ge$ and $>$.

Given $\XB$ a basis of $\XX$, for each $x\in \XX$ and each $1\le p\le \infty$, we note
\begin{align*}
	\left\Vert x\right\Vert_{\ell_{p}}:=&\left\Vert \left(\xx_n^*\left(x\right)\right)_{n\in \NN}\right\Vert_{\ell_{p}}, 
\end{align*}
where the value $\infty$ is allowed. 

Given $\XB$ a basis of $\XX$, $x\in \XX$ and a finite non-empty set $A\subset  \supp\left(x\right)$, the oscillation of $x$ on $A$ (see \cite{DOSZ2009}) is given by 
\begin{align*}
	\osc(x,A)=&\frac{\max_{n\in A}\left\vert \xx_n^*\left(x\right)\right\vert }{\min_{n\in A}\left\vert \xx_n^*\left(x\right)\right\vert}.
\end{align*}
We also set the convention $\osc\left(x, \emptyset\right):=1$. 

Given $\XB$ a basis of $\XX$, and $x, y\in \XX$, we write  $x  	\rhd y $ if 
\begin{align*}
	&\left\vert \xx_k^*\left(x\right)\right\vert \ge \left\vert \xx_n^*\left(y\right)\right\vert &&\forall k\in \supp\left(x\right),\forall n\in \NN. 
\end{align*}
We use the convention $\frac{b}{0}=\infty$ for all $b\in \FF\setminus \{0\}$. 
An absolute constant $C$ is a fixed positive real number (so, in particular, it does not depend on the basis or the space). 

\subsection{Unconditionality}
Unconditional bases are well known bases in the context of Functional Analysis. To define these bases, given $\XB=(\xx_n)_{n\in\mathbb N}$ a basis of $\XX$, a set of indices $A\in\mathbb N^{<\infty}$ and $x\in \XX$, we write the \textit{projection operator} $P_A$ as follows:
\begin{align*}
	&P_A[\XB,\XX](x)=P_A\left(x\right):=\sum_{n\in A}\xx_n^*\left(x\right)\xx_n; &&P_{A^c}\left(x\right):=x-P_A\left(x\right). 
\end{align*}
We say that $\XB$ is \textit{unconditional} if there is $K>0$ such that
$$\Vert P_A(x)\Vert\leq K\Vert x\Vert,\; \forall A\in \NN^{<\infty}, \forall x\in\mathbb X.$$
A weaker condition with an unconditionality-like flavor is nearly unconditionality, introduced by Elton in 1978 in the context of extraction of subsequences with ``good'' behavior from normalized weakly null sequences \cite{Elton1978}.
Let $\XB=\left(\xx_n\right)_{n\in \NN}$ be a basis of a Banach space $\XX$. We define
	\begin{align*}
		&\Cu=\Cu[\XB,\XX]=\left\lbrace x\in \XX \colon \sup_n\left\vert \xx_n^*\left(x\right)\right\vert \le 1\right\rbrace;\\
		& \Cu_0=\Cu_0[\XB,\XX]=\left\lbrace x\in \Cu: \left\vert\xx_n^*\left(x\right)\right\vert\not= \left\vert \xx_j^*\left(x\right)\right\vert  \forall n,j\in \supp\left(x\right): n\not=j\right\rbrace.  
	\end{align*}
Here and in similar cases, we may leave the basis implicit when there is no ambiguity.

For $x\in \Cu$ and $0<t\le 1$, we set
	\begin{align*}
		A\left(x,t\right):=&\left\lbrace n\in \NN: \left\vert \xx_n^*\left(x\right)\right\vert \ge t\right\rbrace. 
	\end{align*}

\begin{defi}
		A basis $\XB$ is {nearly unconditional} if, for each $t\in(0,1]$ there is a positive constant $C_t$ such that for any $f\in \Cu$,
	\begin{equation}\label{nearlyuncdef}
		\Vert P_A(x)\Vert\le C_t\Vert x\Vert,
	\end{equation}
	whenever $A\subseteq A\left(x,t\right)$. The near unconditionality function of the basis $\phi: (0,1]\rightarrow [1,+\infty)$ is defined, for each $t$, as the minimum $C_t$ for which \eqref{nearlyuncdef} holds. 
\end{defi}

	 \subsection{The TGA and Cesàro-summability}

	 Given a basis $\XB$ for $\XX$ and $x\in \XX$, a sequence $\cO=\cO\left(x\right):=\left(k_j\right)_{j\in \NN}$ of positive integers is a \textit{greedy ordering} for $x$ if
	\begin{align*}
		\left\lbrace k_j: 1\le j\le m\right\rbrace\in \cG\left(x,m\right)\qquad\forall m\in \NN. 
	\end{align*}

	 Given a basis $\XB$ for $\XX$, $x\in\XX$ and a greedy ordering $\cO\left(x\right)=\left(k_j\right)_{j\in \NN}$, $G_n^{\cO}\left(x\right)$ denotes the $n$-greedy sum for $x$ with respect to $\cO\left(x\right)$, that is
	\begin{align*}
		G_n^{\cO}\left(x\right)=&\sum_{j=1}^{n}\xx_{k_j}^{*}\left(x\right)\xx_{k_j}, 
	\end{align*}
	whereas $C^{g,\cO}_n\left(x\right)$ denotes the $n$-greedy Cesaro sum for $x$ with respect to $\cO$, that is
	$$
	C^{g,\cO}_n\left(x\right)=\sum_{j=1}^{n}\left(\frac{n+1-j}{n}\right)\xx_{k_j}^*\left(x\right)\xx_{k_j}. 
	$$
	
	 Given a basis $\XB$ for $\XX$, $x\in\XX$ and a greedy ordering $\cO\left(x\right)=\left(k_j\right)_{j\in \NN}$, $\cO_m\left(x\right)$ denotes the greedy ordering $\left(k_{j+m}\right)_{j\in \NN}$ for $x-G_m^{\cO\left(x\right)}\left(x\right)$. 
	\begin{defi}[\cite{KT1999}]\label{definitionQG}
		Given $\XB$ a basis of $\XX$ and $C>0$, we say that $\XB$ is $C$-quasi-greedy if
		\begin{align}
			\left\Vert P_A\left(x\right)\right\Vert\le C\left\Vert x\right\Vert\label{QGconstant}
		\end{align}
		for all $x\in \XX$, $m\in \NN$, and $A\in \cG\left(x,m\right)$. 
		
		Given $\XB$ a basis of $\XX$ and $C>0$, we say that $\XB$ is $C$-suppression quasi-greedy if
		\begin{align*}
			\left\Vert x- P_A\left(x\right)\right\Vert\le C\left\Vert x\right\Vert
		\end{align*}
		for all $x\in \XX$, $m\in \NN$, and $A\in \cG\left(x,m\right)$. 
	\end{defi}
	 
	\begin{defi}
		 Given $\XB$ a basis of $\XX$ and $C>0$, we say that $\XB$ is $C$-Cesaro quasi-greedy ($C$-CQG) if 
		\begin{align*}
			\left\Vert C^{g,\cO}_n\left(x\right)\right\Vert\le C\left\Vert x\right\Vert
		\end{align*}
		whenever $x\in \XX$, $n\in \NN$, and $\cO$ is a  greedy ordering for $x$.
	\end{defi}
	
	\begin{defi}
		 Given $\XB$ a basis and $C>0$, we say that $\XB$ is $C$-de la Vallée-Poussin quasi-greedy ($C$-VPQG) if 
		\begin{align*}
			\left\Vert 2C^{g,\cO}_{2n}\left(x\right)- C^{g,\cO}_n\left(x\right)\right\Vert\le C\left\Vert x\right\Vert
		\end{align*}
		whenever $x\in \XX$, $n\in \NN$, and $\cO$ is a  greedy ordering for $x$. 
	\end{defi}
	 	 
	 The notion of \emph{thresholding bounded} bases was introduced in \cite{DKK2003}. In a sense, it is related to quasi-greediness in the way  near unconditionality is related to unconditionality. 
	 \begin{defi}
	 A basis $\XB$ is {thresholding bounded} if, for each $t\in(0,1]$ there is a positive constant $C_t$ such that for any $f\in \Cu$,
	\begin{equation}\label{thresholdingboundeddef}
		\Vert P_{A\left(x,t\right)}\left(x\right)\Vert\le C_t\Vert x\Vert,
	\end{equation}
	The thresholding boundedness  function of the basis $\theta: (0,1]\rightarrow [1,+\infty)$ is defined, for each $t$, as the minimum $C_t$ for which \eqref{thresholdingboundeddef} holds. 
\end{defi}

\begin{remark}\rm It was proven in \cite{DKK2003} that nearly unconditionality and thresholding boundedness are equivalent notions - unlike unconditionality and quasi-greediness, see \cite{KT1999}.
\end{remark}

	Also, we need the following observation.

	\begin{remark}\rm \label{remarkVPQG}If $\XB$ is a basis of $\XX$, $x\in \XX$, $l, n, m\in \NN$ and $\cO=\left(k_j\right)_{j\in \NN}$ is a greedy ordering for $x$,
		\begin{align}
			&2C_{2n}^{g,\cO}\left(x\right)-C_{n}^{g,\cO}\left(x\right)\nonumber\\
			&= \sum_{j=1}^{2n}\left(\frac{2n+1-j}{n}\right)\xx_{k_j}^*\left(x\right)\xx_{k_j}-\sum_{j=1}^{n}\left(\frac{n+1-j}{n}\right)\xx_{k_j}^*\left(x\right)\xx_{k_j}\nonumber\\
			&=G_{n}^{\cO}\left(x\right)+\sum_{j=n+1}^{2n}\left(\frac{2n+1-j}{n}\right)\xx_{k_j}^*\left(x\right)\xx_{k_j}\nonumber\\&=G_{n}^{\cO}\left(x\right)+\sum_{j=1}^{n}\left(\frac{n+1-j}{n}\right)\xx_{k_{n+j}}^*\left(x\right)\xx_{k_{n+j}}. \label{VPQGb2}
		\end{align}
	\end{remark}
	
	\begin{remark}\label{remarknormalized} Note that if $\XB$ is a basis of $\XX$ and we define a norm 
		\begin{align*}
			\left\Vert x\right\Vert_{\circ}:=&\max\left\lbrace \frac{\left\Vert x\right\Vert}{\alpha_1}, \left\Vert x\right\Vert_{\ell_{\infty}}\right\rbrace,
		\end{align*}
		then,
		\begin{itemize}
			\item for every $x\in\mathbb X$,
			\begin{align*}
				\frac{1}{\alpha_1}\left\Vert x\right\Vert\le   \left\Vert x\right\Vert_{\circ}\le& \alpha_2\left\Vert x\right\Vert.
			\end{align*}
			\item Both $\XB$ and $\XB^*$ are normalized with respect to $\left\Vert \cdot\right\Vert_{\circ}$. 
			\item If $\XB$ is $K$-QG, CQG or VPQG with respect to $\left\Vert \cdot\right\Vert$ the same holds for $\left\Vert \cdot\right\Vert_{\circ}$.
		\end{itemize}
	\end{remark}

\subsection{Indicator sums}
	 One object of interest in the context of greedy approximation is the indicator sum, which is used to define several properties that are useful for the study of greedy-like bases. 
	 
	 Given $\XB$ a basis of $\XX$, $A\subset \NN$, $B\subset A$, with $\left\vert B\right\vert <\infty$, and $\varepsilon\in \EE^{A}$, we define
	\begin{eqnarray*}
		\Ind_{\varepsilon, B}[\XB,\XX]=\Ind_{\varepsilon,B}:=&\sum_{n\in B}\varepsilon_n\xx_n,
	\end{eqnarray*}
	where the sum is zero if $B=\emptyset$. 
	
	\begin{defi}
		 A basis $\XB$ is  $C$-suppression unconditional ($C$ always denotes a positive real number) for constant coefficients (SUCC) if 
		$$
		\|\Ind_{\varepsilon,B}\|\le C\|\Ind_{\varepsilon,A}\|
		$$
		for all sets $A,B\in\mathbb N^{<\infty}$ such that $B\subset A$ and all $\varepsilon\in \EE^{A}$.
	\end{defi}
	
We will consider the following three weak forms of partial unconditionality.  
	
	\begin{defi}
		A basis $\XB$ is  $C$-unconditional for constant coefficients (UCC) if 
		$$
		\|\Ind_{\varepsilon,A}\|\le C\|\Ind_{\varepsilon',A}\|
		$$
		for all  $A\in \mathbb N^{<\infty}$ and all $\varepsilon, \varepsilon'\in \EE^{A}$.
	\end{defi}

	\begin{defi}
		 A basis $\XB$ is  $C$-quasi-greedy for largest coefficients (QGLC) if 
		$$
		\|\Ind_{\varepsilon,A}\|\le C\|\Ind_{\varepsilon,A}+x\|
		$$
		for all  $A\in\mathbb N^{<\infty}$, all $\varepsilon\in \EE^{A}$ and all $x\in \Cu$ with $\supp\left(x\right)\cap A=\emptyset$. 
	\end{defi}
	
	\begin{defi}
		 A basis is $C$-truncation quasi-greedy (TQG) if 
		\begin{align*}
			\min_{n\in A}\left\vert\xx_n^*\left(x\right)\right\vert \Vert \Ind_{\varepsilon\left(x\right), A}\Vert\le& C\left\Vert x\right\Vert
		\end{align*}
		for all $x\in \XX$, $m\in \NN$, $A\in \cG\left(x,m\right)$. 
	\end{defi}
	
Unconditionality for constant coefficients was first defined in \cite{Wo2000}, where it was proved that all quasi-greedy bases have this property. Suppression unconditionality for constant coefficients is an equivalent property introduced in \cite{AABW2021}. On the other hand, quasi-greediness for largest coefficients is a stronger property also introduced in \cite{AABW2021}, which turned out to be equivalent to thresholding boundedness and near unconditionality \cite{AAB2023a}. In Section~\ref{tech}, we will use the latter equivalence in prove that the greedy-like bases we introduce are thresholding bounded.

\section{Caracterization of CQG and VPQG bases}\label{characterization}

In this section, we prove characterizations of Cesaro-quasi-greedy and de la Vallée-Poussin quasi-greedy bases in terms of convergence and pointwise boundedness, along the lines of similar characterizations for quasi-greedy bases proved in \cite{AABW2021} and \cite{Wo2000}. In fact, we will obtain these characterizations from more general results, which may be of use for the study of other variants of greedy-like approximation. We begin with an auxiliary lemma that allows us to go from pointwise boundedness to uniform boundedness. 

\begin{lemma}\label{lemmapointwisebounded}Let $\XB$ be a basis for a Banach space $\XX$, let $A\subset \NN$ be a nonempty set, $K_1$ a positive constant, and $\left(t_n\right)_{n\in \NN_0}$ a sequence of positive numbers. For each $m\in A$, let $\left(b_{m,j}\right)_{1\le j\le m}$ be scalars such that
	\begin{align*}
		&\left\vert b_{m,j}- b_{m,i}\right\vert \le \frac{t_{j-i}}{m}&&\forall 1\le i \le j\le m;\\
		&\left\vert b_{m,j}\right\vert\le K_1 &&\forall 1\le j\le m.
	\end{align*}
	Suppose that for each $x\in \XX$ there is $c_x>0$ such that, for each $m\in A$ there is a greedy ordering $\cO\left(x,m\right)=\left(k_j\right)_{j\in \NN}$ for $x$ for which we have the following bound: 
	\begin{align}
		&\left\Vert \sum_{j=1}^{m}b_{m,j}\xx_{k_j}^*\left(x\right)\xx_{k_j}\right\Vert\le c_x. \label{pointwisebound}
	\end{align}
	Then there is $K_2>0$ such that for every $x\in \XX$, every greedy ordering $\cO\left(x\right)=\left(k_j\right)_{j\in \NN}$  and every $m\in A$, 
	\begin{align}
		&\left\Vert \sum_{j=1}^{m}b_{m,j}\xx_{k_j}^*\left(x\right)\xx_{k_j}\right\Vert\le K_2\left\Vert x\right\Vert. \label{cqgetc}
	\end{align}
\end{lemma}
\begin{proof}
First note that for each $x\in \XX$, we may take $c_x$ to be the minimum positive number for which \eqref{pointwisebound} holds. Now for each $l\in \NN$, let 
	\begin{align*}
		\cF_l:=&\left\lbrace x\in \XX: c_x\le l\right\rbrace. 
	\end{align*}
	Let us prove that $\cF_l$ is a closed set: Fix $x\in \overline{\cF_l}$, and let $\cO=\left(k_j\right)_{j\in \NN}$ be a greedy ordering for $x$. We may assume $\supp\left(x\right)\not=\emptyset$ (else, $c_x=0< l$). Given $j_0\in A$, we consider two cases: 
	\begin{enumerate}
		\item[Case 1:] If $\left\vert \xx_{k_{j_0}}^*\left(x\right)\right\vert>0$, pick $j_1>j_0$ and $\epsilon_1>0$ so that $\left\vert \xx_{k_{j_0}}^*\left(x\right)\right\vert>2\epsilon_1+\left\vert \xx_{k_{j_1}}^*\left(x\right)\right\vert$. Note that $\cG\left(x,j_0,1\right)\subset \cP\left(B\right)$, where $B:=\left\lbrace k_{j}: 1\le j\le j_1\right\rbrace$. Now let $\left(x_n\right)_{n\in \NN}\subset \cF_l$ be a sequence that converges to $x$, and choose $n_1\in \NN$ so that 
		\begin{eqnarray*}
			\left\Vert x-x_n\right\Vert\le \frac{\epsilon_1}{\left(1+\alpha_2\right)}\forall n\ge n_1. 
		\end{eqnarray*}
		Then $\left\vert \xx_j^*\left(x-x_n\right)\right\vert< \epsilon_1$ for all $n\ge n_1$ and all $j\in \NN$, so in particular, 
		\begin{align}
			&\cG\left(x_n,j_0,1\right)\subset \cP\left(B\right) &&\forall n\ge n_1. \label{boundedgreedysets}
		\end{align}
		For each $n\in \NN$, let $\cO\left(x_n,j_0\right)=\left(k_{n,j_0,j}\right)_{j\in \NN}$ be a greedy ordering for $x_n$ such that 
		\begin{align*}
			\left\Vert \sum_{j=1}^{j_0} b_{j_0,j}\xx_{k_{n,j_0,j}}^*\left(x_n\right)\xx_{k_{n,j_0,j}}\right\Vert\le l. 
		\end{align*}
		Then by \eqref{boundedgreedysets}
		\begin{align*}
			&\left\lbrace k_{n,j_0,j}: 1\le j\le j_0\right\rbrace\subset B&&\forall n\ge n_1.
		\end{align*}
		Since $B$ is a finite set, passing to a subsequence we may assume that 
		\begin{align*}
			&k_{n,j_0,j}=k_{l,j_0,j}&&\forall n,l\in \NN\forall 1\le j\le j_0. 
		\end{align*}
		Let us denote that unique value by $k_{j_0,j}$. It follows by the convergence of $\left(x_n\right)_{n\in \NN}$ to $x$ that
		\begin{align}
			\left\Vert \sum_{j=1}^{j_0} b_{j_0,j}\xx_{k_{j_0,j}}^*\left(x\right)\xx_{k_{j_0,j}}\right\Vert\le l, \nonumber
		\end{align}
		and that $\left\lbrace k_{j_0,j}: 1\le j\le i\right\rbrace\in \cG\left(x,i,1\right)$ for all $1\le i\le j_0$. Thus, we can complete $\left(k_{j_0,j}\right)_{1\le j\le j_0}$ to a greedy ordering for $x$, and  it follows that $x\in \cF_l$.\\
		\item[Case 2.] If $\left\vert \xx_{k_{j_0}}^*\left(x\right)\right\vert=0$, let $B=\supp\left(x\right)$. Since $B$ is finite and not empty, it follows that $b:=\min_{n\in B}\left\vert \xx_{n}^*\left(x\right)\right\vert>0$. Let $j_1:=\left\vert B\right\vert$, let $\left(x_n\right)_{n\in \NN}\subset \cF_l$ be a sequence that converges to $x$, and choose $n_1\in \NN$ so that
		\begin{align*}
			&\left\Vert x-x_n\right\Vert\le \frac{b}{2\left(1+\alpha_2\right)}&&\forall n\ge n_1. 
		\end{align*}
		Then $\left\vert \xx_j^*\left(x-x_n\right)\right\vert< \frac{b}{2} $ for all $n\ge n_1$ and all $j\in\NN$, so in particular, 
		\begin{align}
			&\cG\left(x_n,j_1,1\right)=\left\lbrace B\right\rbrace&&\forall n\ge n_1. \label{boundedgreedysets2}
		\end{align}
		For each $n\in \NN$, let $\cO\left(x_n,j_0\right)=\left(k_{n,j_0,j}\right)_{j\in \NN}$ be a greedy ordering for $x_n$ such that 
		\begin{align*}
			\left\Vert \sum_{j=1}^{j_0} b_{j_0,j}\xx_{k_{n,j_0,j}}^*\left(x\right)\xx_{k_{n,j_0,j}}\right\Vert\le l. 
		\end{align*}
		Then by \eqref{boundedgreedysets2}, 
		\begin{align*}
			&\left\lbrace k_{n,j_0,j}: 1\le j\le j_1\right\rbrace= B&&\forall n\ge n_1. 
		\end{align*}
		As before, since $B$ is a finite set we may assume that 
		\begin{align*}
			&k_{n,j_0,j}=k_{l,j_0,j}&&\forall n,l\in \NN\forall 1\le j\le j_1. 
		\end{align*}
		For each $1\le j\le j_1$, let $k_{j_0,j}$ be that unique value. It follows by the convergence of $\left(x_n\right)_{n\in\NN}$ to $x$ that $\left\lbrace k_{j_0,j}: 1\le j\le i\right\rbrace\in \cG\left(x,i,1\right)$ for all $1\le i\le j_1$. Since $B=\supp\left(x\right)$, we can complete $\left(k_{j_0,j}\right)_{1\le j\le j_1}$ to a reordering of $\NN$ in an arbitrary manner and we always obtain a greedy ordering for $x$. Thus, to complete the proof that $x\in \cF_l$, we only need to show that
		\begin{align*}
			\left\Vert \sum_{j=1}^{j_1} b_{j_0,j}\xx_{k_{j_0,j}}^*\left(x\right)\xx_{k_{j_0,j}}\right\Vert\le l. 
		\end{align*}
		This follows from the choice of the $x_n$'s, the triangle inequality and convergence: Considering only $n\ge n_1$ we obtain
		\begin{align*}
			&\left\Vert \sum_{j=1}^{j_1}b_{j_0,j}\xx_{k_{j_0,j}}^*\left(x\right)\xx_{k_{j_0,j}}\right\Vert\underset{j_0\ge j_1}{=}\left\Vert \sum_{j=1}^{j_0} b_{j_0,j}\xx_{k_{n,j_0,j}}^*\left(x\right)\xx_{k_{n,j_0,j}}\right\Vert\\
			&\le \left\Vert \sum_{j=1}^{j_0} b_{j_0,j}\xx_{k_{n,j_0,j}}^*\left(x_n\right)\xx_{k_{n,j_0,j}}\right\Vert+\left\Vert \sum_{j=1}^{j_0}b_{j_0,j}\xx_{k_{j_0,j}}^*\left(x-x_n\right)\xx_{k_{j_0,j}}\right\Vert\\
			&\le l+j_0\alpha_3K_1\left\Vert x-x_n \right\Vert \xrightarrow[\substack{n\to +\infty}]{}l. 
		\end{align*}
	\end{enumerate}
	Since $\cF_l$ is closed for each $l\in \NN$ and $\XX=\bigcup_{l\in \NN}\cF_l$, there is $l_0\in \NN$ with nonempty interior. Using the density of $\langle \XB\rangle$ in $\XX$, we can find $n_0\in \NN$, $x_0\in \langle \xx_n: 1\le n\le n_0\rangle$ and $\epsilon_0>0$ such that the open ball $B\left(x_0,\epsilon_0\right)$ is contained in $\cF_{l_0}$. We may assume by a small perturbation argument that $\supp\left(x_0\right)=\left\lbrace
	\xx_n: 1\le n\le n_0\right\rbrace$ and that 
	\begin{align}
		&\left\vert \xx_j^*\left(x_0\right)\right\vert\not= \left\vert \xx_i^*\left(x_0\right)\right\vert&&\forall 1\le j,i\le n_0: i\not=j. \label{notthesamemodulus}
	\end{align} 
	Let $\XX_0:=\overline{\langle \xx_n: n>n_0\rangle}$. We will first find $K_2>0$ so that \eqref{cqgetc} holds for $x\in \XX_0$.  To this end, first set 
	\begin{align*}
		a_0:=&\min_{1\le n\le n_0}\left\vert \xx_n^*\left(x_0\right)\right\vert;\\
		a_1:=&\left\Vert x_0\right\Vert.
	\end{align*}
	Choose $0<\epsilon_1<\frac{1}{4}\min\left(\epsilon_0, 1, \frac{a_0}{\left(1+\alpha_2\right)}\right)$, and choose $x\in \XX_0$ with $\left\Vert x\right\Vert\le \epsilon_1$. Fix $j_0\in A$ and $\cO=\left(k_j\right)_{j\in \NN}$ a greedy ordering for $x$. We claim that in order to estimate the norm of $\sum_{j=1}^{j_0} b_{j_0,j} \xx_{k_j}^*\left(x\right)\xx_{k_j}$,  we may assume that $k_j>n_0$ for all $1\le j\le j_0$: indeed, if $1\le k_{j_1}\le n_0$ for some $1\le j_1\le j_0$, since $\xx_{k_{j_1}}^*\left(x\right)=0$ it follows that $\left\vert \supp\left(x\right)\right\vert\le j_0$. Thus, there is $i>j_0$ such that $k_i>n_0$ and $k_i\not\in \supp\left(x\right)$. Hence, we define $k'_i=k_{j_1}$ and $k'_{j_1}=k_i$ leaving the rest of the ordering unchanged, and we still have a greedy ordering for $x$, say $\cO'$, such that $$\sum_{j=1}^{j_0} b_{j_0,j} \xx_{k_j}^*\left(x\right)\xx_{k_j}=\sum_{j=1}^{j_0} b_{j_0,j} \xx_{k_j'}^*\left(x\right)\xx_{k_j'}. $$
Then the proof of the claim is completed by iterating the above procedure if necessary. \\	
	Now, if $j_0\le n_0$, then 
	\begin{align}
		&\left\Vert \sum_{j=1}^{j_0} b_{j_0,j} \xx_{k_j}^*\left(x\right)\xx_{k_j}\right\Vert \le K_1n_0\alpha_3\left\Vert x\right\Vert.  \label{smallsupport10}
	\end{align}
	If $j_0>n_0$, 	pick $0<\epsilon<\epsilon_1$ and $y_{\epsilon}\in \XX_0$ so that 
	\begin{align}
		&\left\Vert x-y_{\epsilon}\right\Vert\le \epsilon;\nonumber\\
&\supp\left(x-y_{\epsilon}\right)=\left\lbrace k_1,\dots, k_{j_0}\right\rbrace;\\		
		&\left\vert \xx_{k_j}^*\left(y_{\epsilon}\right)\right\vert> \left\vert \xx_{k_j}^*\left(x\right)\right\vert &&\forall 1\le j\le j_0;\nonumber\\
		&\cG\left(y_{\epsilon},i,1\right)=\left\lbrace \left\lbrace k_{j}: 1\le j\le i\right\rbrace \right\rbrace &&\forall 1\le i\le j_0.\label{onlyoneset}
	\end{align}
	Note that $x_0+y_{\epsilon}\in \cF_{l_0}$, that $\cO$ is a greedy ordering for $y_{\epsilon}$, and the following hold: 
	\begin{align}
		&\left\vert \xx_n^*\left(x_0+y_{\epsilon}\right)\right\vert=\left\vert \xx_n^*\left(x_0\right)\right\vert>\frac{a_0}{2}\ge \left\vert \xx_i^*\left(y_{\epsilon}\right)\right\vert=\left\vert \xx_i^*\left(x_0+y_{\epsilon}\right)\right\vert &&\forall 1\le n\le n_0\forall i>n_0;\label{firstn0}\\
		&\left\vert \xx_{k_{j}}^*\left(x_0+y_{\epsilon}\right)\right\vert=\left\vert \xx_{k_{j}}^*\left(y_{\epsilon}\right)\right\vert>\left\vert \xx_{k_{j+1}}^*\left(y_{\epsilon}\right)\right\vert=\left\vert \xx_{k_{j+1}}^*\left(x_0+y_{\epsilon}\right)\right\vert &&\forall 1\le j\le j_0.\label{secondj0} 
	\end{align}
Since $k_j>n$ for all $1\le j\le j_0$, it follows from the above that 
\begin{align}
&\cG\left(x_0+y_{\epsilon}, i,1\right)\subset \cP\left( \left\lbrace 1,\dots, n_0\right\rbrace\right) &&\forall 1\le i\le n_0 \label{thirdn0}
\end{align}
and 
\begin{align}
\cG\left(x_0+y_{\epsilon}, n_0+i,1\right)=\left\lbrace \left\lbrace 1,\dots, n_0 \right\rbrace\cup \left\lbrace k_{j}: 1\le j\le i\right\rbrace \right\rbrace &&\forall 1\le i\le j_0.  \label{fourthn0}
\end{align}

	Let $\cO'=\left(k_j'\right)_{j\in \NN}$ be a greedy ordering for $x_0+y_{\epsilon}$ such that 
	\begin{align*}
		\left\Vert \sum_{j=1}^{j_0} b_{j_0,j} \xx_{k_{j}'}^*\left(x_0+y_{\epsilon}\right)\xx_{k_{j}'}\right\Vert\le l_0. 
	\end{align*}
	It follows from \eqref{onlyoneset} \eqref{firstn0} and \eqref{secondj0} that $\left(k_{j+n_0}'\right)_{j\in \NN}$ is a greedy ordering for $y_{\epsilon}$ and that 
	\begin{align}
		&k'_{j+n_0}=k_{j}&&\forall 1\le j\le j_0,
	\end{align}
whereas \eqref{notthesamemodulus} and \eqref{thirdn0} entail that for all $1\le j\le n_0$, 
\begin{align*}
& k_j'\le n_0 &&\text{and}&&&\cG\left(x_0+y_{\epsilon}, j, 1\right)=\left\lbrace k_1',\dots, k_j'\right\rbrace.   
\end{align*}	
	Since $\left\Vert x\right\Vert\le \epsilon_1$,  $\left\Vert y_{\epsilon}\right\Vert\le 2\epsilon_1$ and $\left\Vert x-y_{\epsilon}\right\Vert\le \epsilon$, we have 
	\begin{align*}
		&\left\Vert \sum_{j=1}^{j_0} b_{j_0,j} \xx_{k_{j}}^*\left(x\right)\xx_{k_{j}}\right\Vert\le  \epsilon K_1\alpha_3 j_0+\left\Vert \sum_{j=1}^{j_0} b_{j_0,j}\xx_{k_{j}}^*\left(y_{\epsilon}\right)\xx_{k_{j}}\right\Vert
\end{align*}	
and
\begin{align*}
&\left\Vert \sum_{j=1}^{j_0} b_{j_0,j}\xx_{k_{j}}^*\left(y_{\epsilon}\right)\xx_{k_{j}}\right\Vert \le  \left\Vert \sum_{j=1}^{j_0-n_0} b_{j_0,j}\xx_{k_{j+n_0}'}^*\left(y_{\epsilon}\right)\xx_{k_{j+n_0}'}\right\Vert+\left\Vert \sum_{j=j_0-n_0+1}^{j_0} b_{j_0,j}\xx_{k_{j+n_0}'}^*\left(y_{\epsilon}\right)\xx_{k_{j+n_0}'}\right\Vert\\
		&\le 2\epsilon_1\alpha_3K_1n_0+\left\Vert \sum_{j=1}^{j_0-n_0}b_{j_0,j}\xx_{k_{j+n_0}'}^*\left(x_0+y_{\epsilon}\right)\xx_{k_{j+n_0}'}\right\Vert\\
		&=2\alpha_3K_1\epsilon_1n_0+\left\Vert \sum_{j=n_0+1}^{j_0}b_{j_0,j-n_0}\xx_{k_{j}'}^*\left(x_0+y_{\epsilon}\right)\xx_{k_{j}'}\right\Vert\\
		&\le 2\alpha_3K_1 \epsilon_1n_0+\alpha_3K_1n_0a_1+\left\Vert \sum_{j=1}^{n_0} b_{j_0,j} \xx_{k_{j}'}^*\left(x_0+y_{\epsilon}\right)\xx_{k_{j}'}+\sum_{j=n_0+1}^{j_0}b_{j_0,j-n_0}\xx_{k_{j}'}^*\left(x_0+y_{\epsilon}\right)\xx_{k_{j}'}\right\Vert \\
		&\le 3\alpha_3K_1a_1 n_0+\left\Vert \sum_{j=1}^{j_0}b_{j_0,j}\xx_{k_{j}'}^*\left(x_0+y_{\epsilon}\right)\xx_{k_{j}'}\right\Vert+\left\Vert \sum_{j=n_0+1}^{j_0}\left(b_{j_0,j}-b_{j_0,j-n_0}\right)\xx_{k_{j}'}^*\left(x_0+y_{\epsilon}\right)\xx_{k_{j}'}\right\Vert\\
		\le& 3\alpha_3K_1a_1n_0+\alpha_3 t_{n_0}\left(a_1+2\epsilon_1\right) +l_0.
	\end{align*}
	Letting $\epsilon$ tend to $0$ we obtain 
	\begin{align*}
		&\left\Vert \sum_{j=1}^{j_0} b_{j_0,j}\xx_{k_{j}}^*\left(x\right)\xx_{k_{j}}\right\Vert\le 3\alpha_3K_1a_1n_0+\alpha_3 t_{n_0}\left(a_1+2\epsilon_1\right) +l_0.
	\end{align*}
	Since this holds for every $x\in \XX_0$ such that $\left\Vert x\right\Vert\le \epsilon_1$, considering \eqref{smallsupport10} we conclude that \eqref{cqgetc} holds for all $x\in \XX_0$ with $K_2=3\alpha_3K_1a_1n_0+\alpha_3 t_{n_0}\left(a_1+2\epsilon_1\right) +l_0$. 
	It remains to show that this holds for all $x\in \XX$, replacing $K_2$ for some $K_2'$. By an inductive argument, we may assume that $n_0=1$. So, fix $x\in \XX$ with $\xx_1^*\left(x\right)\not=0$ and $j_0\in A$, and let $\cO=\left(k_j\right)_{j\in \NN}$ be a greedy ordering for $x$. Let $y:=x-\xx_1^*\left(x\right)\xx_1$. If $1\not\in\left \lbrace k_j: 1\le j\le j_0\right\rbrace$, then we can complete $\left(k_j\right)_{1\le j\le j_0}$ to a greedy ordering for $y$, so 
	\begin{align*}
		\left\Vert \sum_{j=1}^{j_0} b_{j_0,j}\xx_{k_{j}}^*\left(x\right)\xx_{k_{j}}\right\Vert &=\left\Vert \sum_{j=1}^{j_0} b_{j_0,j}\xx_{k_{j}}^*\left(y\right)\xx_{k_{j}}\right\Vert \le K_2\left\Vert y\right\Vert \\
		&\le K_2\left(1+\alpha_3\right)\left\Vert x\right\Vert. 
	\end{align*}
	On the other hand, if there is $1\le j_1\le j_0$  such that $k_{j_1}=1$, let $\cO'=\left(k_j'\right)_{j\in \NN}$ be defined by 
	\begin{align*}
		k_j':=&
		\begin{cases}
			k_j & \text{ if  } 1\le j<j_1;\\
			k_{j+1}\text{ if }j\ge j_1. 
		\end{cases}
	\end{align*}
	Then $\cO'$ is a greedy ordering for $y$ and $\xx_{k_j'}^*\left(x\right)=\xx_{k_j'}^*\left(y\right)$ for every $j\in\NN$. Hence, 
	\begin{align*}
		&\left\Vert \sum_{j=1}^{j_0}b_{j_0,j}\xx_{k_j}^*\left(x\right)\xx_{k_j}\right\Vert=\left\Vert \sum_{j=1}^{j_1-1}b_{j_0,j}\xx_{k_j'}^*\left(y\right)\xx_{k_j'} +b_{j_0,j_1}\xx_1^*\left(x\right)\xx_1 + \sum_{j=j_1}^{j_0-1}b_{j_0,j+1}\xx_{k_{j}'}^*\left(y\right)\xx_{k_{j}'} \right\Vert\\
		&\le \left\Vert b_{j_0,j_1} \xx_1^*\left(x\right)\xx_1  \right\Vert+ \left\Vert b_{j_0,j_0} \xx_{k'_{j_0}}^*\left(x\right)\xx_{k'_{j_0}}  \right\Vert+ \left\Vert \sum_{j=1}^{j_0}b_{j_0,j}\xx_{k_j'}^*\left(y\right)\xx_{k_j'}\right\Vert +\left\Vert \sum_{j=j_1}^{j_0-1}\left(b_{j_0,j}-b_{j_0,j+1}\right)\xx_{k_j'}^*\left(x\right)\xx_{k_j'}\right\Vert\\
		&\le 2K_1\alpha_3\left\Vert x\right\Vert +K_2\left\Vert y \right\Vert+t_1\alpha_3\left\Vert x\right\Vert\le \left(2K_1\alpha_3+K_2\left(1+\alpha_3\right)+t_1\alpha_3\right)\left\Vert x\right\Vert. 
	\end{align*}
	Hence, \eqref{cqgetc} holds if we replace $K_2$ with $\left(2K_1\alpha_3+K_2\left(1+\alpha_3\right)+t_1\alpha_3\right)$. 
\end{proof}

Next, we give sufficient conditions under which uniform boundedness of the greedy-like sums entails convergence of the greedy-like sums. 

\begin{lemma}\label{lemmaconvergence=boundedness}
	Let $\XB$ be a basis for a Banach space $\XX$, $A\subset \NN$ an infinite set,  and $K_1$ and $K_2$ positive constants. Let $\left(b_{m,j}\right)_{\substack{m\in A\\1\le j\le m}}$ be scalars such that
	\begin{align}
	\begin{split}
		&\left\vert b_{m,j}\right\vert\le K_1 \qquad\forall m\in A, \;\forall 1\le j\le m;\\
		&\lim_{\substack{m\in A_{\ge j}\\m\to+\infty}}b_{m,j}=1 \qquad \forall j\in \NN. \label{convergencebmj}
		\end{split}
	\end{align}
	Suppose that for every $x\in \XX$, every greedy ordering $\cO\left(x\right)=\left(k_j\right)_{j\in \NN}$  and every $m\in A$, 
	\begin{align}
		&\left\Vert \sum_{j=1}^{m}b_{m,j}\xx_{k_j}^*\left(x\right)\xx_{k_j}\right\Vert\le K_2\left\Vert x\right\Vert. \label{cqgetc2}
	\end{align}
	Then $\XB$ is a strong Markushevich basis and, for every $x\in \XX$ and every greedy ordering $\cO\left(x\right)=\left(k_j\right)_{j\in \NN}$  we have
	\begin{align}
		\sum_{j=1}^{m} b_{m,j}\xx_{k_j}^*\left(x\right)\xx_{k_j}\xrightarrow[\substack{ m\to +\infty\\m\in A }]{}x. \label{convergenceequiv}
	\end{align}
\end{lemma}
\begin{proof}
	First, we prove \eqref{convergenceequiv} when $\supp\left(x\right)=\NN$. In that case, given $\cO=\left(k_j\right)_{j\in \NN}$ a greedy ordering for $x$ and $\epsilon>0$, by density we can find $n_1\in \NN$ and $y\in \langle \xx_n: 1\le n\le n_1 \rangle$ so that 
	$\left\Vert x-y\right\Vert<\epsilon$. By a small perturbation argument, we may assume that $\supp\left(y\right)=\left\lbrace 1,\dots,n_1\right\rbrace$ and that $\xx_j^*\left(y\right)\not=\xx_j^*\left(x\right)$ for all $1\le j\le n_1$. Let
	\begin{align*}
		&a:=\min_{1\le j\le n_1}\min\left\lbrace \left\vert \xx_j^*\left(x\right)\right\vert, \left\vert \xx_j^*\left(y\right)\right\vert, \left\vert \xx_j^*\left(y\right)-\xx_j^*\left(x\right)\right\vert\right\rbrace. 
	\end{align*}
	Since $a>0$ and $\supp\left(x\right)$ is infinite, there is $j_1\in \NN$ such that 
	\begin{align*}
		&a>\left\vert \xx_{k_{j_1}}^*\left(x\right)\right\vert >\left\vert \xx_{k_{j_1}+1}^*\left(x\right)\right\vert,&&& \left\lbrace 1,\dots,n_1\right\rbrace\subset \left\lbrace k_j\right\rbrace_{1\le j\le j_1-1}. 
	\end{align*}
	It follows that $\cG\left(x-y,j_1,1\right)=\left\lbrace \left\lbrace k_j: 1\le j\le j_1\right\rbrace\right\rbrace$. Thus, if $\cO'=\left(k_j'\right)_{j\in \NN}$ is a greedy ordering for $x-y$, then 
	\begin{align*}
		&\left\lbrace k_j'\right\rbrace_{1\le j\le j_1}=\left\lbrace k_j\right\rbrace_{1\le j\le j_1}\supset \supp\left(y\right),\end{align*}
	and we can pick one such ordering so that $k_j'=k_j$ for all $j>j_1$. Since $y=\sum_{j=1}^{j_1}\xx_{k_j'}^*\left(y\right)\xx_{k_j'}$, by \eqref{convergencebmj} there is $m_1\in A_{>j_1}$ such that 
	\begin{align*}
		&\left\Vert y-\sum_{j=1}^{m}b_{m,j}\xx_{k_j'}^*\left(y\right)\xx_{k_j'}\right\Vert<\epsilon  &&\forall m\in A_{\ge m_1}.
	\end{align*}
	Now fix $m\in A_{>m_1}$. We have
	\begin{align*}
		&\left\Vert x-\sum_{j=1}^{m}b_{m,j}\xx_{k_j}^*\left(x\right)\xx_{k_j}\right\Vert\\
		&\le \left\Vert x-y\right\Vert+\left\Vert y-\sum_{j=1}^{m}b_{m,j}\xx_{k_j'}^*\left(y\right)\xx_{k_j'}\right\Vert+\left\Vert \sum_{j=1}^{m}b_{m,j}\xx_{k_j'}^*\left(y\right)\xx_{k_j'}-\sum_{j=1}^{m}b_{m,j}\xx_{k_j}^*\left(x\right)\xx_{k_j}\right\Vert\\
		&\le 2\epsilon+\left\Vert \sum_{j=1}^{m}b_{m,j}\xx_{k_j'}^*\left(y-x\right)\xx_{k_j'}\right\Vert+\left\Vert\sum_{j=1}^{m}b_{m,j}\xx_{k_j'}^*\left(x\right)\xx_{k_j'}-\sum_{j=1}^{m}b_{m,j}\xx_{k_j}^*\left(x\right)\xx_{k_j}\right\Vert\\
		& \le 2\epsilon+K_2\epsilon+\left\Vert\sum_{j=1}^{k_1}b_{m,j}\xx_{k_j'}^*\left(x\right)\xx_{k_j'}-\sum_{j=1}^{k_1}b_{m,j}\xx_{k_j}^*\left(x\right)\xx_{k_j}\right\Vert\xrightarrow[\substack{m\to +\infty\\m\in A_{\ge m_1}}]{}2\epsilon+K_2\epsilon.
	\end{align*}
	Hence, we have proven \eqref{convergenceequiv} for $x\in \XX$ when $\supp\left(x\right)=\NN$. \\
	Second, we prove that $\XB$ is a Markushevich basis: Suppose $\xx_n^*\left(x\right)=0$ for all $n\in \NN$. Let $y:=\sum_{n=1}^{\infty}\frac{\xx_n}{2^n}$. Since $y$ and $x+y$ have $\NN$ as their support and both have the same unique greedy ordering, the case proven above gives 
	\begin{align*}
		&y=\lim_{\substack{m\to +\infty\\ m\in A}}\sum_{j=1}^{m}b_{m,j}\xx_j^*\left(y\right)\xx_j=\lim_{\substack{m\to +\infty\\ m\in A}}\sum_{j=1}^{m}b_{m,j}\xx_j^*\left(x+y\right)\xx_j=x+y,
	\end{align*}
	so $x=0$.\\
	Third, suppose that $x$ has finite support. Then by our second step,  $x$ lies in the linear span of $\XB$, so \eqref{convergenceequiv} follows at once from \eqref{convergencebmj}. \\
	Fourth, we prove that $\XB$ is strong. To this end, suppose that $B$ is a proper subset of $\NN$. Since $\XB$ is a Markushevich basis, to prove that 
	\begin{align*}
		\left\lbrace x\in \XX: \xx_n^*\left(x\right)=0\;\forall n\in \NN\setminus B\right\rbrace= \overline{\langle \xx_n: n\in B  \rangle},
	\end{align*}
	we may assume that $B$ is not finite. Fix $x\in \XX\setminus \left\lbrace 0\right\rbrace$, and suppose that $\xx_n^*\left(x\right)=0\;\forall n\in \NN\setminus B$. Let $y:=\sum_{n\not\in \supp\left(x\right)}\frac{\xx_n}{2^n}$. Then $\supp\left(x+y\right)=\NN$, so by the first case, there is a greedy ordering $\cO\left(x+y\right)$ such that
	\begin{align*}
		&x+y=\lim_{\substack{m\to +\infty\\m\in A }}\sum_{j=1}^{m} b_{m,j}\xx_{k_j}^*\left(x+y\right)\xx_{k_j}. 
	\end{align*}
	It follows from \eqref{convergencebmj} and the absolute convergence of the series $\sum_{n\in \NN}\xx_n^*\left(y\right)\xx_n$ that 
	\begin{align*}
		&y=\lim_{\substack{m\to +\infty\\m\in A }}\sum_{j=1}^{m} b_{m,j}\xx_{k_j}^*\left(y\right)\xx_{k_j}. 
	\end{align*}
	Hence, 
	\begin{align*}
		&x=\lim_{\substack{m\to +\infty\\m\in A }}\sum_{j=1}^{m} b_{m,j}\xx_{k_j}^*\left(x\right)\xx_{k_j}=\lim_{\substack{m\to +\infty\\m\in A }}\sum_{\substack{1\le j\le m\\ k_j\in B}} b_{m,j}\xx_{k_j}^*\left(x\right)\xx_{k_j}.
	\end{align*}
	This proves that
	\begin{align*}
		\left\lbrace x\in \XX: \xx_n^*\left(x\right)=0\;\forall n\in \NN\setminus B\right\rbrace\subset \overline{\langle \xx_n: n\in B  \rangle}
	\end{align*}
	Since the reverse inclusion holds for every basis, we have shown that $\XB$ is strong. \\
	We are left to prove \eqref{convergenceequiv} when $\supp\left(x\right)$ is a proper infinite subset $B$ of $\NN$. Since $\XB$ is strong, this follows by a straightforward modification of the argument for the case $B=\NN$, namely we replace $\left\lbrace 1,\dots,n_1\right\rbrace$ by $\left\lbrace 1,\dots,n_1\right\rbrace\cap B$  throughout the proof. 
\end{proof}

\begin{remark}\rm While that is note the focus of this work, we note that Lemmas~\ref{lemmapointwisebounded} and~\ref{lemmaconvergence=boundedness} can be proved for $p$-Banach spaces, with only straightforward modifications. 
\end{remark}
Now we can give the promised characterization of Cesàro quasi-greedy bases. 
\begin{proposition}\label{propositioncarac}Let $\XB$ be a basis for a Banach space $\XX$. The following are equivalent: 
	\begin{enumerate}[\rm (i)]
		\item \label{CQG} $\XB$ is Cesàro quasi-greedy. 
		\item  \label{ConvergenceCQG} For each $x\in \XX$, if $\cO=\left(k_j\right)_{j\in \NN}$ is a greedy ordering for $x$, 
		\begin{eqnarray*}
			x=\lim_{n\to +\infty}C^{g,\cO}_n\left(x\right).
		\end{eqnarray*}
		\item  \label{pointwiseboundedCQG} For each $x\in \XX$ there is $c_x>0$ and  $\cO=\left(k_j\right)_{j\in \NN}$ a greedy ordering for $x$ such that 
		\begin{eqnarray*}
			\left\Vert C^{g,\cO}_n\left(x\right)\right\Vert\le c_x\forall n\in \NN. 
		\end{eqnarray*}
	\end{enumerate}
	If these conditions are met, $\XB$ is a strong Markushevich basis.
\end{proposition}
\begin{proof}
	\ref{pointwiseboundedCQG} $\Rightarrow$ \ref{CQG}: Apply Lemma~\ref{lemmapointwisebounded} with $A=\NN$, $t_n=n$ for all $n\in \NN$, $K_1=1$ and, for every $m\in A$ and every $1\le j\le m$, $b_{m,j}=\frac{m+1-j}{m}$.\\
	\ref{CQG} $\Rightarrow$ \ref{ConvergenceCQG}. If $\XB$ is $K$-CQG, apply Lemma~\ref{lemmaconvergence=boundedness} 
	with $K_2=K$ and $A$, $K_1$ and  $\left\lbrace b_{m,j}\right\rbrace_{\substack{m\in A\\ 1\le j\le m}}$ as before. \\
	Lemma~\ref{lemmaconvergence=boundedness} also entails that $\XB$ is a strong Markushevich basis. 
\end{proof}
A similar characterization holds for de la Vallée-Poussin quasi-greedy bases. 

\begin{proposition}\label{propositioncaracVPQG}Let $\XB$ be a basis for a Banach space $\XX$. The following are equivalent: 
	\begin{enumerate}[\rm (i)]
		\item \label{VPQG3} $\XB$ is de la Vallée Poussin quasi-greedy. 
		\item \label{ConvergenceVPQG3} For each $x\in \XX$, if $\cO=\left(k_j\right)_{j\in \NN}$ is greedy ordering for $x$, 
		\begin{eqnarray*}
			x=\lim_{n\to +\infty}\left(2C^{g,\cO}_{2n}\left(x\right)-C^{g,\cO}_n\left(x\right)\right).
		\end{eqnarray*}
		\item \label{pointwiseboundedVPQG3} For each $x\in \XX$ there is $c_x>0$ and  $\cO=\left(k_j\right)_{j\in \NN}$ a greedy ordering for $x$ such that
		\begin{eqnarray*}
			\left\Vert 2C^{g,\cO}_{2n}\left(x\right)-C^{g,\cO}_n\left(x\right)\right\Vert\le c_x\forall n\in \NN. 
		\end{eqnarray*}
	\end{enumerate}
	If these conditions are met, $\XB$ is a strong Markushevich basis. 
\end{proposition}
\begin{proof}
	\ref{pointwiseboundedVPQG3} $\Rightarrow$ \ref{VPQG3}: Considering Remark~\ref{remarkVPQG}, this follows by an application of Lemma~\ref{lemmapointwisebounded} with $A=2\NN$, $t_n=2n$ for all $n\in \NN$, $K_1=1$ and, for every $m\in \NN$, 
	\begin{align*}
		b_{2m,j}=
		\begin{cases}
			1 & \text{ if }1\le j\le m;\\
			\frac{2m+1-j}{m} \text{ if } m+1\le j\le m. 
		\end{cases}
	\end{align*}
	\ref{VPQG3} $\Rightarrow$ \ref{ConvergenceVPQG3}. If $\XB$ is $K$-VPQG, apply Lemma~\ref{lemmaconvergence=boundedness} 
	with $K_2=K$ and $A$, $K_1$ and  $\left\lbrace b_{2m,j}\right\rbrace_{\substack{m\in \NN\\ 1\le j\le 2m}}$ as before. \\
	Finally, Lemma~\ref{lemmaconvergence=boundedness} also gives that $\XB$ is a strong Markushevich basis. 
\end{proof}

\begin{remark}\label{remarkqg}\rm Note that if apply Lemmas~\ref{lemmapointwisebounded} and~\ref{lemmaconvergence=boundedness} with $A=\NN$ and $b_{m_j}=1$ for all $m\in \NN$ and all $1\le j\le m$, we obtain the classical equivalences for quasi-greedy bases. 
\end{remark}

\section{Some technical results}\label{tech}
In this section, we show that VPQG bases are nearly unconditional, and establish further results that are used in Sections~\ref{subsequences} and~\ref{relation}.

First, we prove near unconditionality (or thresholding boundedness, which is equivalent by \cite[Proposition 4.5]{DKK2003}). To this end, we will use the following results from the literature.
\begin{lemma} \cite[Theorem 2.6]{AAB2023a}, \cite[Proposition 4.5]{DKK2003}\label{lemmaequivqglc} Let $\XB$ be a basis of a Banach space $\XX$. The following are equivalent: 
	\begin{itemize}
		\item $\XB$ is quasi-greedy for largest coefficients. 
		\item $\XB$ is nearly unconditional. 
		\item $\XB$ is thresholding bounded. 
	\end{itemize}
\end{lemma}

\begin{lemma} \cite[Lemma 2.2]{AAB2023a} \cite[Proposition 4.1]{DKK2003}\label{lemmaproduct}
	Let $\XB$ be a nearly unconditional basis of a Banach space $\XX$. Its unconditionality function $\phi$ satisfies the relation
	\begin{align*}
		&\phi\left(ab\right)\le \phi\left(a\right)+\phi\left(b\right)+ \phi\left(a\right)\phi\left(b\right)\le 3\phi\left(a\right)\phi\left(b\right) &&\forall 0<a,b\le 1.
	\end{align*}
\end{lemma}
Now we use Lemma~\ref{lemmaequivqglc} to prove that VPQG bases are nearly unconditional. 

\begin{lemma}\label{lemma: VPQG->NU}Let $\XB$ be a basis of a Banach space $\XX$. If $\XB$ is de la Vallée-Poussin quasi-greedy, it is nearly unconditional. 
\end{lemma}
\begin{proof}
	Pick $C>0$ so that $\XB$ is $C$-VPQG. By Lemma~\ref{lemmaproduct}, we only need to show that $\XB$ is quasi-greedy for largest coefficients. To this end, fix a nonempty set $A\in\mathbb N^{<\infty}$, $\varepsilon\in \EE^{A}$ and $x\in \Cu$ so that $A\cap \supp(x)=\emptyset$. Suppose first that there is $n\in\NN$ such that $2n=\left\vert A\right\vert $, and let $\cS_{A}$ be the set of biyections on $A$. For each $\pi\in \cS_A$, let $\preceq_{\pi}$ be the only total ordering on $A$ such that for $k,l\in A$, $$k\preceq_{\pi} l \Longleftrightarrow \pi(k)\le \pi(l).$$
	For each $1\le j\le 2n$, each $l\in A$ and each $\pi\in \cS_A$, we say that $o(\pi,l)=j$ if and only if 
	$$
	\left\vert k\in A: k\preceq_{\pi} l\right\vert=j
	$$
	(in other words, if $l$ is the $j$-th. element of $A$ under the ordering $\preceq_{\pi}$).\\
	For each $\pi\in \cS_A$, pick a greedy ordering $\mathcal{O}_{\pi}=\left(k_{j}^{\pi}\right)_{j\in \NN}$ for $\Ind_{\varepsilon, A}+x$ so that $\left\lbrace k_j^{\pi}\right\rbrace_{1\le j\le 2n}=A$ and, for each $1\le j\le 2n$ and each $l\in A$, 
	$$k_{j}^{\pi}=l\Longleftrightarrow o(\pi,l)=j.$$
	For each $\pi\in \cS_{A}$, let $A_{\pi}:=\left\lbrace k_j^{\pi}\right\rbrace_{1\le j\le n}$. Since 
	\begin{align*}
		G_{n}^{\cO_{\pi}}\left(\Ind_{\varepsilon, A}+x\right)=&\Ind_{\varepsilon, A_{\pi}},
	\end{align*}
	by Remark~\ref{remarkVPQG} and hypothesis we have
	\begin{align}
		&\left\Vert \Ind_{\varepsilon, A_{\pi}}+\sum_{j=1}^{n}\left(\frac{n+1-j}{n}\right)\varepsilon_{k_{n+j}^{\pi}}\xx_{k_{n+j}^{\pi}}\right\Vert\nonumber\\
		=&\left\Vert G_{n}^{\cO_{\pi}}\left(\Ind_{\varepsilon, A}+x\right)+\sum_{j=1}^{n}\left(\frac{n+1-j}{n}\right)\varepsilon_{k_{n+j}^{\pi}}\xx_{k_{n+j}^{\pi}}\right\Vert\le C\left\Vert \Ind_{\varepsilon, A}+x\right\Vert.  \label{bound1VPQG}
	\end{align}
	For each $l\in A$, we have
	\begin{align*}
		&\xx_l^*\left(\sum_{\pi\in \cS_A}\left(\Ind_{\varepsilon, A_{\pi}}+\sum_{i=1}^{n}\left(\frac{n+1-i}{n}\right)\varepsilon_{k_{n+i}^{\pi}}\xx_{k_{n+i}^{\pi}}\right)\right)\\
		&\left\vert\left\lbrace \pi\in  \cS_A: o\left(\pi,l\right)\le n \right\rbrace \right\vert \varepsilon_l+ \xx_l^*\left(\sum_{j=n+1}^{2n}\sum_{\substack{\pi\in \cS_A\\o\left(\pi,l\right)=j }}\left(\sum_{i=1}^{n}\left(\frac{n+1-i}{n}\right)\varepsilon_{k_{n+i}^{\pi}}\xx_{k_{n+i}^{\pi}}\right)\right)\\
		&=n\left(2n-1\right)!\varepsilon_l+\sum_{j=1}^{n}\sum_{\substack{\pi\in \cS_A\\o\left(\pi,l\right)=n+j }}\left(\sum_{i=1}^{n}\left(\frac{n+1-i}{n}\right)\varepsilon_{k_{n+i}^{\pi}}\xx_l^*\left(\xx_{k_{n+i}^{\pi}}\right)\right)\\
		&=n\left(2n-1\right)!\varepsilon_l+\sum_{j=1}^{n}\sum_{\substack{\pi\in \cS_A\\o\left(\pi,l\right)=n+j }}\left(\frac{n+1-j}{n}\right)\varepsilon_l\\
		&=n\left(2n-1\right)!\varepsilon_l+\left(2n-1\right)!\varepsilon_l\sum_{j=1}^{n}\left(\frac{n+1-j}{n}\right)=\varepsilon_l\left(2n-1\right)!\left(n+\frac{n+1}{2}\right).
	\end{align*}
	Hence, 
	\begin{align*}
		\sum_{\pi\in \cS_A}\left(\Ind_{\varepsilon, A_{\pi}}+\sum_{i=1}^{n}\left(\frac{n+1-i}{n}\right)\varepsilon_{k_{n+i}^{\pi}}\xx_{k_{n+i}^{\pi}}\right)=&\left(\frac{\left(2n-1\right)!\left(3n+1\right)}{2}\right)\Ind_{\varepsilon,A}. 
	\end{align*}
	From this and \eqref{bound1VPQG}, by the triangle inequality we obtain
	\begin{align*}
		\left\Vert\left(\frac{\left(2n-1\right)!\left(3n+1\right)}{2}\right)\Ind_{\varepsilon,A}\right\Vert\le& \sum_{\pi\in \cS_A}\left\Vert \Ind_{\varepsilon, A_{\pi}}+\sum_{i=1}^{n}\left(\frac{n+1-i}{n}\right)\varepsilon_{k_{n+i}^{\pi}}\xx_{k_{n+i}^{\pi}}\right\Vert\\
		\le& \left(2n\right)!C\left\Vert \Ind_{\varepsilon,A}+x\right\Vert. 
	\end{align*}
	Hence, 
	\begin{align*}
		\left\Vert\Ind_{\varepsilon, A}\right\Vert\le& \frac{4nC}{3n+1}\left\Vert \Ind_{\varepsilon,A}+x\right\Vert\le \frac{4C}{3} \left\Vert \Ind_{\varepsilon,A}+x\right\Vert.
	\end{align*}
	Now suppose that $\left\vert A\right\vert =2n+1$ for some $n\in \NN$, pick $A_0\subset A$ with $\left\vert A_0\right\vert =2n$ and $j_0\in A\setminus A_0$, and let $y:=x+\varepsilon_{j_0}\xx_{j_0}$. By the previous case, 
	\begin{align*}
		\left\Vert \Ind_{\varepsilon, A}\right\Vert\le& \left\Vert \Ind_{\varepsilon,A_0}\right\Vert+\left\Vert \xx_{j_0}\right\Vert \le    \frac{4C}{3} \left\Vert \Ind_{\varepsilon,A_0}+y\right\Vert+\alpha_1\left\vert \xx_{j_0}^*\left(\Ind_{\varepsilon, A}+x\right)\right\vert\\
		\le&\left( \frac{4C}{3}+\alpha_1\alpha_2\right) \left\Vert \Ind_{\varepsilon,A}+x\right\Vert.
	\end{align*}
	Finally, if $\left\vert A\right\vert=1$, we have
	\begin{align*}
		\left\Vert \Ind_{\varepsilon, A}\right\Vert\le& \alpha_1\alpha_2\left\Vert \Ind_{\varepsilon, A}+x\right\Vert. 
	\end{align*}
	Comparing the above upper bounds, we conclude that $\XB$ is $\left( \frac{4C}{3}+\alpha_1\alpha_2\right)$-QGLC. 
\end{proof}
\begin{remark}\label{remarkUCC}\rm Note that Lemma~\ref{lemma: VPQG->NU} entails (by convexity) that if $\XB$ is $K$-VPQG, it is $K_2$-UCC, with 
	$$
	K_2= 2\kappa \left(\alpha_1\alpha_2+  \frac{4K}{3}\right). 
	$$
\end{remark}
For our study of VPQG and CQG bases, it is convenient to define a function that is equivalent to $\phi$, but in which we maximize over signs instead of projections. More precisely, if $\XB$ is a nearly unconditional basis, we define $\phi_u\left(t\right)$ as the minimum $C\ge 1$ such that
\begin{align*}
	\left\Vert \sum_{n\in A}a_n\xx_n^*\left(x\right)\xx_n \right\Vert\le&C \left\Vert x\right\Vert, 
\end{align*}
where $0<t\le 1$ and the following hold: 
\begin{align*}
	&A\in \NN^{<\infty};&& A\subset \supp\left(x\right);\\
	&\left\vert a_n\right\vert\le 1\;\forall n\in A;&&\min_{n\in A}\left\vert \xx_n^*\left(x\right)\right\vert\ge  t\left\Vert x\right\Vert_{\ell_{\infty}}. 
\end{align*}
\begin{remark}\label{remarkphiu}\rm By convexity and definitions, it is easy to check that 
	\begin{align*}
		\phi\left(t\right)\le&\phi_u\left(t\right)\le 2\kappa\phi\left(t\right) \qquad\forall 0<t\le 1. 
	\end{align*}Combining this with Lemma~\ref{lemmaproduct}, we obtain
	\begin{align*}
		\phi_u\left(ab\right)\le& 6\kappa\phi\left(a\right)\phi\left(b\right)\qquad \forall  0<a,b\le 1.
	\end{align*}
\end{remark}

The following elementary lemma will be useful in the sequel.

\begin{lemma}\label{lemmasortofqg}Let $\XB$ be a basis of $\XX$. For all $x\in \XX$, $m\in \NN$, and $\cO=\left(k_j\right)_{j\in \NN}$ a greedy ordering for $x$, we have
	\begin{align}
		&2C_{2m}^{g,\cO}\left(x\right)-C_{m}^{g,\cO}\left(x\right)=G_{m}^{\cO}\left(x\right)+ C_{m}^{g,\cO_m}\left(x-G_m^{\cO}\left(x\right)\right). \label{remark12v}
	\end{align}
	Hence, If $\XB$ is $K$-VPQG, we have
	\begin{align*}
		\left\Vert G_m^{\cO}\left(x\right)+C_{m}^{g,\cO_m}\left(x-G_m^{\cO}\left(x\right)\right)\right\Vert\le& K\left\Vert x\right\Vert
	\end{align*}
	and then 
	\begin{align*}
		\left\Vert x-G_m^{\cO}\left(x\right)-C_{m}^{g,\cO_m}\left(x-G_m^{\cO}\left(x\right)\right)\right\Vert\le& \left(K+1\right)\left\Vert x\right\Vert.
	\end{align*}
\end{lemma}

\begin{proof}
	By Remark~\ref{remarkVPQG}, 
	\begin{align*}
		&2C_{2m}^{g,\cO}\left(x\right)-C_{m}^{g,\cO}\left(x\right)\\
		&=G_{m}^{\cO}\left(x\right)+\sum_{j=1}^{m}\left(\frac{m+1-j}{m}\right)\xx_{k_{m+j}}^*\left(x-G_m^{\cO}\left(x\right)\right)\xx_{k_{m+j}}.
	\end{align*}
	Since $\cO_m$ is a greedy ordering for $x-G_m^{\cO}\left(x\right)$, we obtain 
	\begin{align}
		&2C_{2m}^{g,\cO}\left(x\right)-C_{m}^{g,\cO}\left(x\right)=G_{m}^{\cO}\left(x\right)+ C_{m}^{g,\cO_m}\left(x-G_m^{\cO}\left(x\right)\right), \label{simple4}
	\end{align}
	so 
	\begin{align*}
		\left\Vert G_m^{\cO}\left(x\right)+C_{m}^{g,\cO_m}\left(x-G_m^{\cO}\left(x\right)\right)\right\Vert\le& K\left\Vert x\right\Vert, 
	\end{align*}
	and then
	\begin{align*}
		\left\Vert x-G_m^{\cO}\left(x\right)-C_{m}^{g,\cO_m}\left(x-G_m^{\cO}\left(x\right)\right)\right\Vert\le& \left(K+1\right)\left\Vert x\right\Vert.
	\end{align*}
\end{proof}
\color{black}
The next lemma allows us  to associate to each VPQG basis a threshold function that is more general than $\phi$ or $\phi_u$; this function is used in several of the proofs in this paper.

\begin{lemma}\label{lemmasortoftqg3}Let $\XB$ be a VPQG basis of a Banach space $\XX$, and let $\preceq$ be the order defined on $\left[0,\infty\right)\times \left[0,1\right]$ by 
	\begin{align*}
		\left(t,s\right)\preceq \left(t',s'\right)\Longleftrightarrow t\le t' \text{ and } s\ge s'. 
	\end{align*}
	Define a function $\Psi:\left[0,\infty\right)\times \left(0,1\right]\rightarrow \left[1,\infty \right]$ as follows: $\Psi\left(t,s\right)$ is the minimum $1\le C\le \infty$ for which 
	\begin{align}
		\left\Vert \sum_{j\in \supp(y)} a_j \xx_j^*(y)\xx_j\right\Vert\le& C\left\Vert  x+y+z\right\Vert \label{sumtobound3}
	\end{align}
	whenever the following conditions hold: 
	\begin{enumerate}[\rm 1)]
		\item \label{iL3a} $x$, $y$ and $z$ have pairwise disjoint support.
		\item \label{iiL3a} $x\rhd y\rhd z$.
		\item \label{iiiL3a} $\left\vert \supp\left(x\right)\right\vert\le t\left\vert \supp(y)\right\vert$. 
		\item $y$ has finite support, and $\osc\left(y,\supp\left(y\right)\right) \le s^{-1}$. 
		\item \label{iiiiL3a} $\left\vert a_j\right\vert\le 1$ for all $j\in \supp(y)$. 
	\end{enumerate}
	Then $\Psi$ is non-decreasing for the ordering $\preceq$ and finite. Moreover, if $\XB$ is $K$-VPQG, then 
	\begin{align}
		&\Psi\left(t,s\right)\le 2\alpha_3+ \Psi\left(1,s\right) \left(K+1\right)^{2+\log_2t} &&\forall 0<s\le 1\;\forall t>1,\nonumber 
	\end{align} 
	and there is an absolute constant $C$ such that 
	\begin{align*}
		& \Psi\left(1,s\right)\le C \alpha_1\alpha_2\alpha_3 \frac{1}{s^3}\left(K\phi\left(\frac{1}{4}\right)\phi\left(s\right)\right)^4 &&\forall 0<s\le 1.
	\end{align*}
\end{lemma}
(note that $\Psi\left(0,s\right)=\phi_u\left(s\right)<\infty$ by Lemma~\ref{lemma: VPQG->NU}). 
\begin{proof}
	Clearly, $\Psi$ is clearly non-decreasing for the ordering $\preceq$, so to prove the lemma we only need to prove that $\Psi$ is finite and that the upper bounds in the statement hold. 
	To this end, first we prove the result for $t=1$, which will also prove that $\Psi$ is finite when $t\le 1$. Note that: 
		\begin{itemize}
		\item If $x=0$, we can use $\phi_u\left(s\right)$ as an upper bound in \eqref{sumtobound3}, so we may assume from now on that $x\not=0$. 
		\item The case $y=0$ is immediate, so we may also assume that $y\not=0$. 
		\item $\XB$ is a Markushevich basis by Proposition~\ref{propositioncaracVPQG}, so $\supp(x)\not=\emptyset$ and $\supp(y)\not=\emptyset$. From \ref{iiL3a} we also obtain that $x$ has finite support.
		\item By a standard density argument (or see \cite[Lemma 3.7]{BBG2024}), we may assume that $z$ has finite support.
	\end{itemize}
	Fix $0<s\le 1$,  $x, y, z\in \XX$, and scalars $\left(a_j\right)_{j\in\supp\left(y\right)}$ so that \ref{iL3a}-\ref{iiiiL3a} hold. Let $A:=\supp\left(x\right)$,  $B=\supp\left(y\right)$, $m:=\left\vert A\right\vert$ and $n:=\left\vert B\right\vert$. Since $t=1$, $m\le n$. Suppose first $n+m=2m_1$ with $m_1\in\NN$, $m\le m_1\le n$. If $n\le 25$, then 
	\begin{align}
		\left\Vert \sum_{j\in B}a_j\xx_j^*\left(y\right)\xx_j\right\Vert\le& 25\alpha_3\left\Vert x+y+z\right\Vert. \label{firstboundhere}
	\end{align}
	If $m_1\le 12$, then $m\le 12$ so 
	\begin{align}
		\left\Vert \sum_{j\in B}a_j\xx_j^*\left(y\right)\xx_j\right\Vert\le&\phi_u\left(s\right)\left\Vert y+z\right\Vert\le \left(12\alpha_3+1\right)\phi_u\left(s\right)\left\Vert x+y+z\right\Vert. \label{secondboundhere}
	\end{align}
	If $m_1>12$ and $n>25$, let $u_0:=x+y+z$, and pick $\cO=\left(k_j\right)_{j\in \NN}$ a greedy ordering for $u_0$ so that $A=\{k_1,\dots, k_m\}$ and $B=\left\lbrace k_{m+1}, \dots, k_{m+n}\right\rbrace$. Let $A_0:=\{k_1,\dots, k_{m_1}\}$ and $B_0:=\{k_{m_1+1}, \dots, k_{2m_1}\}$.  Then $B_0\subset B$ because $m_1\ge m$ and $2m_1=n+m$. Given that $\cO_{m_1}$ is a greedy ordering for $u_0-G_{m_1}^{\cO}\left(u_0\right)$, applying Lemma~\ref{lemmasortofqg} to $u_0$ we obtain 
	\begin{align}
		&\left\Vert u_1:=\sum_{j=1}^{m_1}\left(\frac{j-1}{m_1}\right)\xx_{k_{m_1+j}}^*\left(y\right)\xx_{k_{m_1+j}}+z\right\Vert \nonumber\\
		&=\left\Vert u_0-G_{m_1}^{\cO}\left(u_0\right)-C_{m_1}^{\cO_{m_1}}\left(u_0-G_{m_1}^{\cO}\left(u_0\right)\right)\right\Vert\le \left(K+1\right)\left\Vert u_0\right\Vert\nonumber\\
		=&\left(K+1\right)\left\Vert x+y+z\right\Vert. \label{newbound1}
	\end{align}
	Let $B_1:=\left\lbrace k_{m_1+j}\right\rbrace_{\frac{m_1}{3}<j\le m_1}$. Then $\emptyset \subsetneq B_1\subset B_0\subset B$. For each $\frac{m_1}{3}< j\le m_1$, we have $\frac{j-1}{m_1}\ge \frac{1}{4}$ (because $m_1> 12$). Since $\left\Vert u_1\right\Vert_{\ell_{\infty}}\le \left\Vert y\right\Vert_{\ell_{\infty}}$, it follows that 
	\begin{align*}
		\frac{\left\Vert u_1\right\Vert_{\ell_{\infty}}}{\min_{l\in B_1}\left\vert \xx_l^*\left(u_1\right)\right\vert}\le& \frac{\left\Vert y\right\Vert_{\ell_{\infty}}}{\frac{1}{4}\min_{l\in B_1}\left\vert \xx_l^*\left(y\right)\right\vert} \le \frac{4}{s}.
	\end{align*}
	Hence, 
	\begin{align}
		&\sup_{\left\vert b_l\right\vert\le 1\;\forall l\in B_1} \left\Vert \sum_{l\in B_1}b_l\xx_l^{*}\left(y\right)\xx_l\right\Vert \le 4 \sup_{\left\vert b_l\right\vert\le 1\;\forall l\in B_0} \left\Vert \sum_{l\in B_0}\frac{1}{4} b_l\xx_l^{*}\left(y\right)\xx_l\right\Vert\\
		&\le 4 \sup_{\substack{\left\vert b_j\right\vert\le 1\;\forall 1\le j \le m_1\\ \left\vert \varepsilon\right\vert =1}} \left\Vert \sum_{j=1}^{m_1}\left(\frac{j-1}{m_1}\right)b_j\xx_{k_{m_1+j}}^*\left(y\right)\xx_{k_{m_1+j}}+\varepsilon z \right\Vert \nonumber\\
		&= 4 \sup_{\left\vert b_j\right\vert\le 1\;\forall 1\le j \le m_1 } \left\Vert \sum_{j=1}^{m_1}\left(\frac{j-1}{m_1}\right)b_j\xx_{k_{m_1+j}}^*\left(y\right)\xx_{k_{m_1+j}}+z \right\Vert\nonumber\\
		&\le 4\phi_u\left(\frac{s}{4}\right)\left\Vert u_1\right\Vert\le 4\phi_u\left(\frac{s}{4}\right)\left(K+1\right)\left\Vert x+y+z\right\Vert\nonumber\\
		&\le 48 \phi\left(\frac{1}{4}\right)\left(K+1\right)\phi\left(s\right)\left\Vert x+y+z\right\Vert,\label{secondbound}
	\end{align}
	where we used Remark~\ref{remarkphiu} for the last inequality.
	
	Now let 
	\begin{align*}
		u_2:=&x+y+z-P_{B_1}\left(y\right)+\left\vert \xx_{k_{m_1+1}}^*\left(y\right)\right\vert\Ind_{B_1}. 
	\end{align*}
	Given that for each $l\in B_1$, 
	\begin{align*}
		\left\vert \xx_{l}^*\left(y\right)\right\vert \ge& s \max_{j\in B_0}\left\vert \xx_{j}^*\left(y\right)\right\vert=s \left\vert \xx_{k_{m_1+1}}^*\left(y\right)\right\vert, 
	\end{align*}
	it follows from \eqref{secondbound} and the triangle inequality that 
	\begin{align}
		\left\Vert u_2\right\Vert\le& \left(1+48\left(1+\frac{1}{s}\right)\phi\left(\frac{1}{4}\right)\phi\left(s\right)\left(K+1\right)\right)\left\Vert x+y+z\right\Vert.\label{thirdbound2}
	\end{align}
	Let $n_1:=\left\vert B_1\right\vert$, and define $\cO':=\left(d_j\right)_{j\in \NN}$ by 
	\begin{align*}
		d_j:=&
		\begin{cases}
			k_j & \text{if } 1\le j\le m_1;\\
			k_j & \text{if } 2m_1+1\le j;\\
			k_{j+m_1-n_1} &\text{if } m_1+1\le j\le m_1+n_1;\\
			k_{j-n_1} &\text{if } m_1+n_1+1\le j\le 2m_1. 
		\end{cases}
	\end{align*}
	One can check that $\cO'$ is a greedy ordering for $u_2$ for which each element of $B_1$ comes before each element of $B_0\setminus B_1$, and the ordering outside $B_0$ is the same as it is for $\cO$.  \\
	Let $y_2:=P_{B}\left(u_2\right)$. Note that 
	$y_2=y-P_{B_1}\left(y\right)+\left\vert \xx_{k_{m_1+1}}^*\left(y\right)\right\vert\Ind_{B_1}$, so 
	\begin{align*}
		\left\Vert y_2\right\Vert_{\ell_{\infty}}= &\left\Vert y\right\Vert_{\ell_{\infty}};\\
		\min_{l\in B}\left\vert \xx_l^*\left(y_2\right)\right\vert \ge &\min_{l\in B}\left\vert \xx_l^*\left(y\right)\right\vert,
	\end{align*}
	so conditions \ref{iL3a}-\ref{iiiiL3a} still hold if we replace $y$ by $y_2$, and we still have $\left\vert \supp\left(y_2\right)\right\vert=n$. Hence, if we set 
	\begin{align*}
		B_2:=&\left\lbrace d_{m_1+j}\right\rbrace_{\frac{m_1}{3}<j\le m_1},
	\end{align*}
	the same argument used to obtain
	\eqref{secondbound}  gives
	\begin{align}
		\sup_{\left\vert b_l\right\vert\le 1\;\forall l\in B_2} \left\Vert \sum_{l\in B_2}b_l\xx_l^{*}\left(y_2\right)\xx_l\right\Vert\le& 48\phi\left(\frac{1}{4}\right)\phi\left(s\right)\left(K+1\right)\left\Vert u_2\right\Vert \nonumber. 
	\end{align}
	Combining this upper bound with \eqref{thirdbound2} we get 
	\begin{align}
		&\sup_{\left\vert b_l\right\vert\le 1\;\forall l\in B_2} \left\Vert \sum_{l\in B_2}b_l\xx_l^{*}\left(y_2\right)\xx_l\right\Vert
		\le  K_1\left\Vert x+y+z\right\Vert, \label{longbound1}
	\end{align}
	where 
	\begin{align*}
		K_1=48\phi\left(\frac{1}{4}\right)\phi\left(s\right)\left(K+1\right)\left(1+48\left(1+\frac{1}{s}\right)\phi\left(\frac{1}{4}\right)\phi\left(s\right)\left(K+1\right)\right).
	\end{align*}
	Since 
	\begin{align*}
		&m_1>n_1=|B_1|=|B_2|>\frac{m_1}{2};\\
		&B_0=\left\lbrace d_{m_1+j} \right\rbrace_{1\le j\le m_1};\\
		&B_1=\left\lbrace d_{m_1+j} \right\rbrace_{1\le j\le n_1};\\
		&B_2=\left\lbrace  d_{m_1+j} \right\rbrace_{m_1+1-n_1 \le j\le m_1},
	\end{align*}
	it follows that $B_0\setminus B_1\subset B_2$. Thus, from \eqref{secondbound},  \eqref{longbound1} by the triangle inequality we get 
	\begin{align}
		&\sup_{\left\vert b_l\right\vert\le 1\;\forall l\in B_0} \left\Vert \sum_{l\in B_0}b_l\xx_l^{*}\left(y\right)\xx_l\right\Vert\le 2K_1\left\Vert x+y+z\right\Vert.\label{inB0}
	\end{align}
	It remains to find an upper bound for the norm of $\sum_{l\in B\setminus B_0}b_l\xx_l^*\left(y\right)\xx_l$, with $\left\vert b_l\right\vert\le 1$. To this end, assume that $B\not=B_0$, and notice that $\left\vert B\setminus B_0\right\vert=n-m_1=m_1-m< m_1=\left\vert B_0\right\vert$, so there is $B_3\subset B_0$ with $\left\vert B_3\right\vert=\left\vert B\setminus B_0\right\vert$. Set 
	\begin{align*}
		y_3:=&y-P_{B_3}\left(y\right)+\left\vert \xx_{k_{m+1}}^*\left(y\right)\right\vert \Ind_{B_3}
	\end{align*}
	Since $\osc\left(y,B\right)\le s^{-1}$ and 
	\begin{align*}
		\left\vert \xx_{k_{m+1}}^*\left(y\right)\right\vert=&\max_{l\in B}\left\vert \xx_{{l}}^*\left(y\right)\right\vert,
	\end{align*}
	it follows that conditions~\ref{iL3a}-\ref{iiiiL3a} still hold when we replace $y$ with $y_3$, leaving $x$ and $z$ as before.  Now pick $\cO''=\left(k_j''\right)_{j\in \NN}$ a greedy ordering for $u_3:=x+y_3+z$ so that $A=\left\lbrace k_1'',\dots, k_m''\right\rbrace$, $B=\left\lbrace k_{m+1}'', \dots, k_{m+n}''\right\rbrace$ and $B_3=\left\lbrace k_{m+1}'', \dots, k_{m+n-m_1}''\right\rbrace$ (which is possible because $B_3$ is a greedy set for $y_3$ of cardinality $n-m_1$). Let $A_0'':=\left\lbrace k_{j}'': 1\le j\le m_1\right\rbrace$ and $B_0'':=\left\lbrace k_{m_1+j}'': 1\le j\le m_1\right\rbrace$. Given that $n+m=2m_1$ and $B_3$ precedes $B\setminus B_3$ in the ordering $\cO''$, we have $B_0''=B\setminus B_3$. Since $\xx_l^*\left(y_3\right)=\xx_l^*\left(y\right)$ for all $l\in B_0''$, the same argument used to obtain \eqref{inB0} gives
	\begin{align}
		&\sup_{\left\vert b_l\right\vert\le 1\;\forall l\in B_0''} \left\Vert \sum_{l\in B_0''}b_l\xx_l^{*}\left(y\right)\xx_l\right\Vert\le 2K_1\left\Vert x+y_3+z\right\Vert.\label{inB0''}
	\end{align}
	As
	\begin{align*}
		&\left\Vert x+y_3+z\right\Vert\le \left\Vert x+y+z\right\Vert+\left\Vert P_{B_3}\left(y\right)\right\Vert +\left\vert \xx_{k_{m+1}}^*\left(y\right)\right\vert \left\Vert \Ind_{B_3}\right\Vert\\
		&\underset{\substack{\eqref{inB0}\\ B_3\subset B_0}}{\le}\left(1+2K_1\right) \left\Vert x+y+z\right\Vert+\frac{1}{s}\left\vert\xx_{k_{m+n}}^*\left(y\right)\right\vert \left\Vert \Ind_{B_3}\right\Vert\\
		&\le \left(1+2K_1\right) \left\Vert x+y+z\right\Vert+\frac{1}{s}\sup_{\left\vert b_l\right\vert\le 1\;\forall l\in B_3} \left\Vert \sum_{l\in B_3}b_l\xx_l^{*}\left(y\right)\xx_l\right\Vert\\
		&\le \left(1+2K_1+\frac{2K_1}{s}\right)\left\Vert x+y+z\right\Vert,
	\end{align*}
	from \eqref{inB0}, \eqref{inB0''}, the triangle inequality and the fact that $B=B_0\cup B_0''$ we obtain 
	\begin{align*}
		\left\Vert \sum_{l\in B}a_l\xx_l^{*}\left(y\right)\xx_l\right\Vert\le&\sup_{\left\vert b_l\right\vert\le 1\;\forall l\in B} \left\Vert \sum_{l\in B}b_l\xx_l^{*}\left(y\right)\xx_l\right\Vert\\
		\le& \sup_{\left\vert b_l\right\vert\le 1\;\forall l\in B_0} \left\Vert \sum_{l\in B_0}b_l\xx_l^{*}\left(y\right)\xx_l\right\Vert+\sup_{\left\vert b_l\right\vert\le 1\;\forall l\in B_0''} \left\Vert \sum_{l\in B_0''}b_l\xx_l^{*}\left(y\right)\xx_l\right\Vert \\
		\le& 4K_1\left(1+K_1+\frac{K_1}{s}\right) \left\Vert x+y+z\right\Vert\le \frac{12K_1^2}{s} \left\Vert x+y+z\right\Vert.
	\end{align*}
	Combining the above inequality with the estimates for the cases $m_1\le 12$ and $n\le 25$, we have obtained 
	\begin{align*}
		\left\Vert \sum_{l\in B}a_l\xx_l^{*}\left(y\right)\xx_l\right\Vert\le& \frac{C\alpha_3}{s^3}\left(K\phi\left(\frac{1}{4}\right)\phi\left(s\right)\right)^4\left\Vert x+y+z\right\Vert,
	\end{align*}
	where $C>1$ is an absolute constant. \\
	Now suppose $n+m=2m_1-1$ with $m_1\in \NN$, $m\le m_1\le n$, and choose $j_1\not\in \supp\left(x+y+z\right)$. Let 
	\begin{align*}
		&y_4:=y+\max_{l\in B}\left\vert\xx_l^*\left(y\right)\right\vert \xx_{j_1};&&u_4:=x+y_4+z;
		&&&B_4:=B\cup\left\lbrace j_4\right\rbrace.
	\end{align*} 
	Then applying the result obtained in the previous case, we deduce that 
	\begin{align*}
		\left\Vert \sum_{l\in B}a_l\xx_l^{*}\left(y\right)\xx_l\right\Vert\le& \sup_{\left\vert b_l\right\vert\le 1\;\forall l\in B_4} \left\Vert \sum_{l\in B_4}b_l\xx_l^{*}\left(y_4\right)\xx_l\right\Vert\\
		\le& \frac{C\alpha_3}{s^3}\left(K\phi\left(\frac{1}{4}\right)\phi\left(s\right)\right)^4\left\Vert x+y_4+z\right\Vert\\
		\le& \left(\alpha_1\alpha_2+1\right)\frac{C\alpha_3}{s^3}\left(K\phi\left(\frac{1}{4}\right)\phi\left(s\right)\right)^4\left\Vert x+y+z\right\Vert\\
		\le& \frac{2C\alpha_1\alpha_2\alpha_3}{s^3}K\left(\phi\left(\frac{1}{4}\right)\phi\left(s\right)\right)^4\left\Vert x+y+z\right\Vert.
	\end{align*}
	Since $x, y, z$ are arbitrary, this completes the proof of the case $t=1$ (and then, $t\le 1$), with
	\begin{align*}
		&\Psi\left(r,s\right)\le \Psi\left(1,s\right)\le \frac{2C\alpha_1\alpha_2\alpha_3}{s^3}\left(K\phi\left(\frac{1}{4}\right)\phi\left(s\right)\right)^4&&\forall 0\le r\le 1\;\forall 0<s\le 1. 
	\end{align*}
	
	To prove the result for $t>1$, fix $0<s\le 1$, $x, y, z\in \XX$, and scalars $\left(a_j\right)_{j\in\supp\left(y\right)}$, and let $A:=\supp\left(x\right)$,  $B=\supp\left(y\right)$, $m:=\left\vert A\right\vert$ and $n:=\left\vert B\right\vert$. If $n\le 2$, then
	\begin{align*}
		\left\Vert \sum_{j\in \supp(y)} a_j \xx_j^*(y)\xx_j\right\Vert\le&\phi_u\left(s\right)\left\Vert y+z\right\Vert\le \left(2\alpha_3+1\right)\phi_u\left(s\right) \left\Vert x+y+z\right\Vert.
	\end{align*}
	If $m\le n$, then we apply the result for $t=1$, so 
	\begin{align*}
		\left\Vert \sum_{j\in \supp(y)} a_j \xx_j^*(y)\xx_j\right\Vert\le&\Psi\left(1,s\right)\left\Vert x+y+z\right\Vert.
	\end{align*}
	If $m>n>2$, let $\cO=\left(k_j\right)_{j\in \NN}$ be a greedy ordering for $u_0:=x+y+z$ so that $A=\left\lbrace k_j\right\rbrace_{1\le j\le m}$ and $B=\left\lbrace k_j\right\rbrace_{m+1\le j\le m+n}$, let $d_1:=\floor{\frac{m}{2}}$, and set 
	\begin{align*}
		u_1:=u_0-G_{d_1}^{\cO}\left(u_0\right)-C_{d_1}^{g,\cO_{d_1}}\left(u_0-G_{d_1}^{\cO}\left(u_0\right)\right). 
	\end{align*}
	Then 
	\begin{align*}
		&\left\Vert u_1\right\Vert_{\ell_{\infty}}\le \left\Vert u_0\right\Vert_{\ell_{\infty}}; &&\left\Vert u_1\right\Vert \underset{\text{Lemma}~\ref{lemmasortofqg}}{\le}\left(K+1\right) \left\Vert u_0\right\Vert.
	\end{align*}
	Also, since
	\begin{align*}
		&\xx_l^*\left(u_1\right)=\xx_l^*\left(u_0\right)&& \forall l\not\in A
	\end{align*}
	and
	\begin{align*}
		&m_1:=\left\vert A_1:=\left\lbrace l\in \NN: \left\vert\xx_l^*\left(u_1\right)\right\vert> \left\vert \xx_{k_{m+1}}^*\left(u_1\right)\right\vert=\left\Vert y\right\Vert_{\ell_{\infty}}\right\rbrace\right\vert\le \left\vert A\setminus \left\lbrace k_j\right\rbrace_{1\le j\le d_1}\right\vert\\
		&=m-d_1\le \frac{m+1}{2},
	\end{align*}
	we can write $u_1=x_1+y_1+z_1$, where conditions \ref{iL3a}-\ref{iiiiL3a} hold when substituting $x_1$, $y_1$ and $z_1$ and $t_1$ for $x$, $y$, $z$ and $t$ respectively, where 
	\begin{align*}
		&\supp\left(x_1\right)=A_1,&&B_1:=\supp\left(y_1\right)\supset B=\supp\left(y\right),&&&t_1=\left(\frac{m+1}{2m}\right)t.
	\end{align*}	
	Let $n_1:=\left\vert B_1\right\vert$ (note that $n_1\ge n$ and $m_1<m$, and that $\xx_l^*\left(u_1\right)=\xx_l^*\left(u_0\right)$ for all $l\in B$). If $m_1\le n_1$, we stop and apply the result for $t=1$, obtaining 
	\begin{align*}
		\left\Vert \sum_{j\in \supp\left(y\right)} a_j \xx_j^*(y)\xx_j\right\Vert\le& \sup_{\left\vert b_l\right\vert\le 1\;\forall l\in B_1} \left\Vert \sum_{l\in B_1}b_l\xx_l^{*}\left(y_1\right)\xx_l\right\Vert\le \Psi\left(1,s\right)\left\Vert u_1\right\Vert\\
		\le& \Psi\left(1,s\right)\left(K+1\right)\left\Vert x+y+z\right\Vert. 
	\end{align*}
	Otherwise, we repeat the procedure  defining $u_2$ in terms of $u_1$, and setting $x_2$,  $y_2$, $z_2$, $A_2$, $B_2$,  $m_2$ and $n_2$ as before. Then 
	\begin{align*}
		m_2\le& \frac{m_1+1}{2}\le \frac{m}{4}+\frac{1}{4}+\frac{1}{2}.
	\end{align*}
	We continue in this manner so that if we have obtained $u_l$ with $m_l\le n_l$, we stop; otherwise, we construct $m_{l+1}$. Given that $m_{l+1}<m_l$ while $n_l\ge n$, we stop after a finite number of steps. Let $l_0$ be the step at which we stop, so that $m_{l_0}\le n_{l_0}$. Given that
	\begin{align*}
		&m_{l_0}\le n_{l_0};&& \supp\left(y\right)\subset \supp\left(y_{l_0}\right);&&\xx_j^*\left(y\right)=\xx_j^*\left(y_{l_0}\right)\;\forall j\in B,
	\end{align*}
	an iterative application of Lemma~\ref{lemmasortofqg} gives
	\begin{align}
		\left\Vert \sum_{j\in \supp\left(y\right)} a_j \xx_j^*\left(y\right)\xx_j\right\Vert\le& \sup_{\left\vert b_j\right\vert\le 1\;\forall j\in B_{l_0}} \left\Vert \sum_{l\in B_{l_0}}b_j\xx_j^{*}\left(y_{l_0}\right)\xx_l\right\Vert\le \Psi\left(1,s\right)\left\Vert u_{l_0}\right\Vert\nonumber\\
		\le& \Psi\left(1,s\right)\left(K+1\right)^{l_0}\left\Vert x+y+z\right\Vert. \label{almost5}
	\end{align}
	It remains to find an upper bound for $l_0$ so that the statement holds.  If $l_0=1$ there is nothing to do; else, note that for each $1\le l\le l_0$, by construction we have 
	\begin{align*}
		&\left\vert \supp\left(x_l\right)\right\vert=\left\vert A_l\right\vert=m_l\le \frac{m}{2^l}+\sum_{j=1}^{l}\frac{1}{2^j}\le \frac{m}{2^l}+1;&&n_{l}\ge n. 
	\end{align*}
	Thus, 
	\begin{align*}
		\frac{m}{2^{l}}+1\le& n\Longrightarrow m_l\le n_l. 
	\end{align*}
	Hence, the antecedent of the above implication does not hold for any $l\le l_0-1$ (else, the procedure would have stopped before $l_0$). Therefore, 
	\begin{align*}
		\frac{m}{2^{l_0-1}}+1> n, 
	\end{align*}
	so 
	\begin{align*}
		l_0-1\le& \log_2\left(\frac{m}{n-1}\right)\le \log_2\left(\frac{m}{n}\right)+1. 
	\end{align*}
	Given that $m\le tn$, it follows that $l_0\le 2+\log_2t$.
\end{proof}
In the rest of the paper, $\Psi$ will always denote the function of Lemma~\ref{lemmasortoftqg3}.
\begin{remark}\rm \label{remarktqg}
	If $\XB$ is truncation quasi-greedy, by \cite[Theorem 1.5, Corollary 2.5]{AAB2023b} we have
	\begin{align*}
		\Psi\left(t,s\right)& \lesssim 1+\log\left(\frac{1}{s}\right)\qquad\forall 0\le t<+\infty\;\forall 0<s\le 1, 
	\end{align*}
	and we could simplify the proofs that use the function $\Psi$ by removing the variable $t$. 
\end{remark}

\section{VPQG bases and convergence through subsequences}\label{subsequences}

In \cite{O2018}, a weaker variant of quasi-greediness was introduced, where the definition is relaxed so that \eqref{QGconstant} only needs to hold for $A\in \cG\left(x,n_k,1\right)$ for a fixed sequence of natural numbers $\bn=\left(n_k\right)_{k\in\NN}$. Bases with this property are called $\bn$-quasi-greedy, and this condition is equivalent to the convergence of the greedy sums of cardinality in $\bn$ \cite[Theorem 2.1] {O2018}. While the study of convergence through subsequences has focused in the case where sequence $\bn$ is the same for each $x\in \XX$ (see also \cite{BB2020} or \cite{BB2024}, for example), a variant where there is convergence through a variable sequence $\bn\left(x\right)=\left(n_k\left(x\right)\right)_{k\in\NN}$ was also addressed (see \cite[Question 6.5]{O2018}, \cite[Example 4.1]{BB2020}). In this section, we study the boundedness and convergence of the greedy sums through variable sequences in the case of VPQG bases. To obtain our results, Lemma~\ref{lemmasortoftqg3} will play a key role.
\begin{lemma}\label{lemmagaps1}Let $\XB$ be a $K$-VPQG basis of $\XX$, $A\subset \NN$ a nonempty set, $\bn=\left(n_l\right)_{l\in A}\subset\NN$ a strictly increasing sequence, and 
	\begin{align*}
		E=&\left\lbrace n_l+j \right\rbrace_{\substack{l\in A\\0\le j\le n_l}}. 
	\end{align*}
	If $x\in \XX$ and $\cO=\left(k_j\right)_{j\in \NN}$ is a greedy ordering for $x$ such that 
	\begin{align*}
		M:=&\sup_{l\in A}\osc\left(x, \left\lbrace k_j\right\rbrace_{n_l+1\le j\le 2n_l}\right)<+\infty,
	\end{align*}
	then for each $n\in E$ we have 
	\begin{align}
		&\left\Vert G_{n}^{\cO}\left(x\right)\right\Vert\le \left(K+2\Psi\left(1,\frac{1}{M}\right)\right)\left\Vert x\right\Vert.\label{bound17}
	\end{align}
	Also, if $A$ is infinite, we have
	\begin{align}
		\lim_{\substack{n \to +\infty\\ n\in E}}G_{n}^{\cO}\left(x\right)=x. \label{conv17}
	\end{align}
\end{lemma}
\begin{proof}
	Fix $l\in \NN$. By Lemma~\ref{lemmasortoftqg3} and hypothesis, 
	\begin{align}
		\sup_{\left\vert b_j\right\vert\le 1\;\forall 1\le j\le n_l} \left\Vert \sum_{j=1}^{n_l}b_j\xx_{k_{n_l+j}}^{*}\left(x \right)\xx_{k_{n_l+j}}\right\Vert\le& \Psi\left(1,\frac{1}{M}\right)\left\Vert x\right\Vert. \label{uncondbound6}
	\end{align}
	Thus, by Lemma~\ref{lemmasortofqg}, 
	\begin{align*}
		\left\Vert G_{n_l}^{\cO}\left(x\right)\right\Vert\le& \left\Vert G_{n_l}^{\cO}\left(x\right)+C_{n_l}^{g,\cO_{n_l}}\left(x- G_{n_l}^{\cO}\left(x\right)\right) \right\Vert+\left\Vert C_{n_l}^{g,\cO_{n_l}}\left(x- G_{n_l}^{\cO}\left(x\right)\right)\right\Vert\\
		\le& \left(K+\Psi\left(1,\frac{1}{M}\right)\right)\left\Vert x\right\Vert. 
	\end{align*}
	Now fix $n_l< n\le 2n_l$. We have
	\begin{align*}
		\left\Vert G_{n}^{\cO}\left(x\right)\right\Vert\le& \left\Vert G_{n_l}^{\cO}\left(x\right)\right\Vert +\left\Vert\sum_{1\le j\le n-n_l}\xx_{k_{n_l+j}}^*\left(x\right)\xx_{k_{n_l+j}}\right\Vert\\ \le&\left(K+2\Psi\left(1,\frac{1}{M}\right)\right)\left\Vert x\right\Vert.
	\end{align*}
	This completes the proof of \eqref{bound17}. To prove \eqref{conv17}, we may assume $A=\NN$. If $x$ has finite support, then $x\in \langle \XB\rangle$ because $\XB$ is a Markushevich basis, so there is nothing to prove. If $x$ has infinite support, fix $\epsilon>0$, and choose $m_0\in \NN$ so that, for each $m\ge m_0$, 
	\begin{align*}
		\left\Vert u_m:=x-G_{m}^{\cO}\left(x\right)-C_{m}^{g,\cO_{m}}\left(x- G_{m}^{\cO}\left(x\right)\right) \right\Vert\le &\frac{\epsilon}{ 4\left(1+\Psi\left(1,\frac{1}{M}\right)\right)}
	\end{align*}
	(which is possible by Proposition~\ref{propositioncaracVPQG} and Lemma~\ref{lemmasortofqg}). Choose $l_0\in \NN_{>m_0}$ so that $k_{n_{l}}>2 m_0$ for all $l\ge l_0$. Given $l\ge l_0$, set
	\begin{align*}
		A_l:=&\left\lbrace k_{j}\right\rbrace_{1\le j\le n_l};\\
		B_l:=&\left\lbrace k_{n_l+j}\right\rbrace_{1\le j\le n_l};\\
		A_l':=&\left\lbrace i\in \NN: \left\vert \xx_i^*\left(u_{m_0}\right)\right\vert>\left\vert \xx_{k_{n_l+1}}^*\left(x\right)\right\vert\right\rbrace;\\
		B_l':=&\left\lbrace i\in \NN: \left\vert \xx_{k_{2n_l}}^*\left(x\right)\right\vert\le  \left\vert \xx_i^*\left(u_{m_0}\right)\right\vert\le \left\vert \xx_{k_{n_l+1}}^*\left(x\right)\right\vert\right\rbrace.
	\end{align*}
	Then $A_l'\subset A_l$ because $\cO$ is a greedy ordering for $x$ and 
	\begin{align*}
		&\left\vert \xx_i^*\left(u_{m_0}\right)\right\vert\le  \left\vert \xx_i^*\left(x\right)\right\vert&&\forall i\in \NN, 
	\end{align*}
	whereas $B_l\subset B_l'$ because
	\begin{align}
		&\xx_i^*\left(u_{m_0}\right)=\xx_i^*\left(x\right)&&\forall i>2m_0. \label{equal17}
	\end{align}
	Hence, $\left\vert A_l'\right\vert \le \left\vert B_l'\right\vert$. From \eqref{equal17},  Lemma~\ref{lemmasortoftqg3} and hypothesis we deduce that
	\begin{align}
		\sup_{\left\vert b_j\right\vert\le 1\;\forall 1\le j\le n_l} \left\Vert \sum_{j=1}^{n_l}b_j\xx_{k_{n_l+j}}^{*}\left(x \right)\xx_{k_{n_l+j}}\right\Vert\le& \sup_{\left\vert b_i\right\vert\le 1\;\forall i\in B_l'} \left\Vert \sum_{i\in B_l' }b_i\xx_{i}^{*}\left(u_{m_0} \right)\xx_{i}\right\Vert\nonumber\\
		\le& \Psi\left(1,\frac{1}{M}\right)\left\Vert u_{m_0}\right\Vert.\label{bound17b}
	\end{align}
	Thus, 
	\begin{align}
		\left\Vert x-G_{n_l}^{\cO}\left(x\right)\right\Vert\le& \left\Vert u_{n_l}\right\Vert+\left\Vert C_{n_l}^{g,\cO_{n_l}}\left(x-G_{n_l}^{\cO}\left(x\right)\right)\right\Vert\nonumber\\
		\le&  \left\Vert u_{n_l}\right\Vert+\sup_{\left\vert b_j\right\vert\le 1\;\forall 1\le j\le n_l} \left\Vert \sum_{j=1}^{n_l}b_j\xx_{k_{n_l+j}}^{*}\left(x \right)\xx_{k_{n_l+j}}\right\Vert\nonumber\\
		\le& \left\Vert u_{n_l}\right\Vert+\Psi\left(1,\frac{1}{M}\right)\left\Vert u_{m_0}\right\Vert<\epsilon.\label{bound17c}
	\end{align}
	Now if $n_l< n\le 2n_l$, then 
	\begin{align*}
		\left\Vert x-G_{n}^{\cO}\left(x\right)\right\Vert\le& \left\Vert x-G_{n_l}^{\cO}\left(x\right)\right\Vert+\left\Vert \sum_{j=1}^{n-n_l}\xx_{k_{n_l+j}}^*\left(x\right)\xx_{k_{n_l+j}}\right\Vert\\
		\underset{\eqref{bound17b},\eqref{bound17c}}{\le}& \left\Vert u_{n_l}\right\Vert+2\Psi\left(1,\frac{1}{M}\right)\left\Vert u_{m_0}\right\Vert<\epsilon.
	\end{align*}
\end{proof}

For VPQG bases, we can also give an estimate for how often we can find uniformly bounded greedy sums for a given $x$. More precisely, we have the following result. 
\begin{lemma}\label{lemmaeachlog}Let $\XB$ be a $K$-VPQG basis of $\XX$. Given $x\in \XX$, $\cO$ a greedy ordering for $x$ and $n\in \NN$, there is an integer $0\le i\le \log_2n+3$ such that 
	\begin{align}
		\left\Vert G_{2^in}^{\cO}\left(x\right)\right\Vert\le& \left(\alpha_1\alpha_2+K+2\Psi\left(1,\frac{1}{4}\right)\right)\left\Vert x\right\Vert. \nonumber
	\end{align}
\end{lemma}
\begin{proof}
	Suppose the result does not hold for some $x\in \XX$ and $n\in \NN$, and let
	\begin{align*}
		D:=&\left\lbrace m\in \NN: \osc\left(x, \left\lbrace k_{m+1}, \dots, k_{2m}\right\rbrace\right)\le 4\right\rbrace;\\
		J:=&\left\lbrace i \in\NN_0:  2^ln\not \in D\quad\forall 0\le l\le i\right\rbrace;\\
		\beta:=&\sup J;\\
		i_1:=&\ceil{\log_2n}+1.
	\end{align*}
	For each $m\in D$ and each $m\le l\le 2m$, by Lemma~\ref{lemmagaps1} we have
	\begin{align*}
		\left\Vert G_{l}^{\cO}\left(x\right)\right\Vert\le& \left(K+2\Psi\left(1,\frac{1}{4}\right)\right)\left\Vert x\right\Vert. 
	\end{align*} 
	Hence, our assumption entails that $\left\lbrace 0,1,\dots, i_1+1\right\rbrace\subset J$, so $i_1<\beta\le +\infty$. Note that $J$ is either $\NN_0$ or a finite interval of $\NN_0$, and that for each $i\in J$, 
	\begin{align}
		&\left\vert \xx_{k_{2^in}}^*\left(x\right)\right\vert \le \frac{1}{4^i}\left\vert \xx_{k_{n}}^*\left(x\right)\right\vert.\label{telescopic1}
	\end{align}
	We consider two cases: 
	\begin{enumerate}
		\item[1)] \label{betafin}If $\beta\in \NN$, then $J=\left\lbrace 0,1,\dots, \beta\right\rbrace$, and by \eqref{telescopic1} we have
		\begin{align*}
			\left\Vert G_{2^{\beta+1}n}^{\cO}\left(x\right)- G_{2^{i_1}n}^{\cO}\left(x\right) \right\Vert\le& \sum_{i=i_1}^{\beta-1}\sum_{j=2^{i}n+1}^{2^{i+1}n}\left\Vert \xx_{k_j}^*\left(x\right)\xx_{k_j}\right\Vert+\sum_{j=2^{\beta}n+1}^{2^{\beta+1}n}\left\Vert \xx_{k_j}^*\left(x\right)\xx_{k_j}\right\Vert\\
			\le& \sum_{i=i_1}^{\beta-1}\sum_{j=2^{i}n+1}^{2^{i+1}n}\frac{1}{4^{i}}\left\Vert \xx_{k_n}^*\left(x\right)\xx_{k_j}\right\Vert+\sum_{j=2^{\beta}n+1}^{2^{\beta+1}n}\left\Vert \xx_{k_{2^{\beta}n}}^*\left(x\right)\xx_{k_j}\right\Vert \\
			\le& \sum_{i=i_1}^{\beta-1}\sum_{j=2^{i}n+1}^{2^{i+1}n}\frac{\alpha_1\alpha_2\left\Vert x\right\Vert}{4^{i}}+\sum_{j=2^{\beta}n+1}^{2^{\beta+1}n}\frac{1}{4^{\beta}}\left\Vert \xx_{k_{n}}^*\left(x\right)\xx_{k_j}\right\Vert \\
			\le& n\alpha_1\alpha_2\left\Vert x\right\Vert\left(\left(\sum_{i=i_1}^{\beta-1}\sum_{j=2^{i}n+1}^{2^{i+1}n}\frac{1}{2^{i}}\right)+ \frac{2^{\beta}n}{4^{\beta}}\right)\underset{\beta>i_1\ge 1+\log_2 n}{\le}\alpha_1\alpha_2\left\Vert x\right\Vert.
		\end{align*}
		Given that $2^{\beta+1}n\in D$, by Lemma~\ref{lemmagaps1} we have
		\begin{align*}
			\left\Vert G_{2^{\beta+1}n}^{\cO}\left(x\right) \right\Vert\le&\left(K+2\Psi\left(1,\frac{1}{4}\right)\right)\left\Vert x\right\Vert.
		\end{align*}
		Hence, 
		\begin{align*}
			\left\Vert G_{2^{i_1}n}^{\cO}\left(x\right)\right\Vert\le& \left(\alpha_1\alpha_2+K+2\Psi\left(1,\frac{1}{4}\right)\right)\left\Vert x\right\Vert, 
		\end{align*}
		which contradicts the assumption. 
		\item[2)] \label{betainf}If $\beta=+\infty$ then $J=\NN_0$, and as before it follows  from \eqref{telescopic1} that 
		\begin{align*}
			\left\Vert x-G_{2^{i_1}n}^{\cO}\left(x\right)\right\Vert\le& \sum_{i=i_1}^{\infty}\sum_{j=2^{i}n+1}^{2^{i+1}n}\left\Vert \xx_{k_j}^*\left(x\right)\xx_{k_j}\right\Vert\le \sum_{i=i_1}^{\infty}\frac{2^in}{4^i}\alpha_1\alpha_2\left\Vert x\right\Vert\\
			=& n\alpha_1\alpha_2\left\Vert x\right\Vert\sum_{i=i_1}^{\infty}\frac{1}{2^i}\le \alpha_1\alpha_2\left\Vert x\right\Vert,
		\end{align*}
		so 
		\begin{align*}
			\left\Vert G_{2^{i_1}n}^{\cO}\left(x\right)\right\Vert\le& \left(1+\alpha_1\alpha_2\right)\left\Vert x\right\Vert,
		\end{align*}
		a contradiction. 
	\end{enumerate}
\end{proof}
\color{black}

For some of our results below, we will use the following definitions (see \cite{O2018} and  \cite{BB2020}). 
\begin{defi}\label{definitiondifferentgaps}
	Let $\bn=\left(n_i\right)_{i\in \NN}$ be a strictly increasing sequence of natural numbers. The \emph{gaps} of the sequence are the quotients $\frac{n_{i+1}}{n_i}$, and we say that $\bn$ has arbitrarily large gaps if
	$$\limsup_{i\rightarrow +\infty}\frac{n_{i+1}}{n_i}=+\infty.$$
	\noindent On the other hand, for $l\in\mathbb N_{\ge 2}$, we say that $\bn$ has $l$-bounded quotient gaps if 
	$$\frac{n_{i+1}}{n_i}\leq l,$$
	for all $i\in\mathbb N$, and we say that it has bounded gaps if it has $l$-bounded gaps for some $l\in\mathbb N_{\ge 2}$.\\
	From now on, $\bn$ and $\bd$ will always denote strictly increasing sequences of positive integers $\left(n_i\right)_{i\in \NN}$ and $\left(d_i\right)_{i\in \NN}$ respectively.
\end{defi}
For our next result, we need the following elementary lemma. 
\begin{lemma}\label{lemmaell1}Let $\XB$ be a basis of $\XX$, $x\in \XX$, $\cO=\left(k_j\right)_{j\in \NN}$ a greedy ordering for $\XX$, $\epsilon>0$, $l\in \NN_{\ge 2}$ and $\bn$ a sequence with $l$-bounded gaps.  Then  $\left\Vert x\right\Vert_{\ell_1}<+\infty$ if at least one of the following conditions hold: 
	\begin{enumerate}[\rm i)]
		\item \label{ni+1}There is $i_0\in \NN$ such that  
		\begin{align*}
			&\osc\left(x,\left\lbrace k_{{n_i}+1},\dots, k_{n_{i+1}}\right\rbrace\right)\ge  l+\epsilon &&\forall i\ge i_0.
		\end{align*}
		\item \label{2ni}There is $i_0\in \NN$ such that 
		\begin{align*}
			&\osc\left(x,\left\lbrace k_{{n_i}+1},\dots, k_{2n_{i}}\right\rbrace\right)\ge  2l+\epsilon &&\forall i\ge i_0.
		\end{align*}
	\end{enumerate}
\end{lemma}
\begin{proof}
	First, note that the oscillation is not affected by renormings, so by Remark~\ref{remarknormalized} we may assume that both $\XB$ and $\XB^*$ are normalized. \\
	Second, by considering only the sequence $\left(n_{i}\right)_{i\ge i_0}$ and renumbering if necessary, we may assume $i_0=1$ in both \ref{ni+1} and \ref{2ni}. \\
	Third, suppose that \ref{ni+1} holds, and note that for each $i\in \NN_{\ge 2}$, 
	\begin{align*}
		\sum_{m=n_{1+i}+1}^{n_{i+2}}\left\Vert \xx_{k_m}^*\left(x\right)\xx_{k_m}\right\Vert\le&  \left\vert \xx_{k_{n_{1+i}+1}}^*\left(x\right)\right\vert n_{1+i+1}\\
		\le& \left(\frac{l}{l+\epsilon}\right)\left\vert \xx_{k_{n_{1+i-1}+1}}^*\left(x\right)\right\vert n_{1+i}, 
	\end{align*}
	and 
	\begin{align*}
		\left\vert \xx_{k_{n_{1+i-1}+1}}^*\left(x\right)\right\vert n_{1+i}\le&  \left(\frac{l}{l+\epsilon}\right)\left\vert \xx_{k_{n_{1+i-2}+1}}^*\left(x\right)\right\vert n_{i}.
	\end{align*}
	Thus, iterating we get 
	\begin{align*}
		&\sum_{m=n_{1+i}+1}^{{n_{i+2}}}\left\Vert \xx_{k_m}^*\left(x\right)\xx_{k_m}\right\Vert\le \left(\frac{l}{l+\epsilon}\right)^i \left\vert \xx_{k_{n_{1}}+1}^*\left(x\right)\right\vert n_{2}&&\forall i\ge 2,
	\end{align*}
	so the sum over all $i\in \NN_{\ge 2}$ is finite. Since $\XB$ is normalized, it follows that $\left\Vert x\right\Vert_{\ell_1}<+\infty$ . \\
	Finally, suppose \ref{2ni} holds and, for every $i\in \NN$, choose $f\left(i\right)>i$ so that 
	\begin{align*}
		2n_i\le n_{f\left(i\right)}< 2l n_i. 
	\end{align*}
	Let us consider the subsequence $\bd=\left(d_j=n_{i_j}\right)_{j\in \NN}$ of $\bn$ defined as follows: $i_1=1$, $i_2=f\left(i_1\right)$ and generally, $i_{j+1}=f\left(i_j\right)$. Then for all $j\in \NN$, we have
	\begin{align*}
		2d_j\le d_{j+1}<2l d_j. 
	\end{align*}
	It follows that $\bd$ has $2l$-bounded gaps and that for each $j\in \NN$, 
	\begin{align*}
		\osc\left(x,\left\lbrace k_{d_j+1}, \dots, k_{d_{j+1}}\right\rbrace\right)\ge \osc\left(x,\left\lbrace k_{d_j+1}, \dots, k_{2d_{j}}\right\rbrace\right)\ge 2l+\epsilon. 
	\end{align*}
	Hence, \ref{ni+1} holds for $\bd$ and $2l$ instead of $\bn$ and $l$ respectively, and the proof is complete. 
\end{proof}

\begin{cor}\label{corollarydaptivegaps}Let $\XB$ be a VPQG basis of $\XX$. Given $x\in \NN$, $\cO$ a greedy ordering for $x$, and $\bn$ a sequence with bounded gaps, there is a subsequence $\bd$ of $\bn$ such that 
	\begin{align*}
		x=&\lim_{k\to +\infty}G_{d_k}^{\cO}\left(x\right). 
	\end{align*}
\end{cor}
\begin{proof}
	Pick $l\in \NN_{\ge 2}$ so that $\bn$ has $l$-bounded gaps, and $\epsilon>0$. Let
	\begin{align*}
		D:=&\left\lbrace i\in \NN: \osc\left(x,\left\lbrace k_{n_i+1}, \dots, k_{2n_i}\right\rbrace \right)\le 2l+\epsilon \right\rbrace.
	\end{align*}
	If $D\in \NN^{<\infty}$,  Lemma~\ref{lemmaell1} entails that $\left\Vert x\right\Vert_{\ell_1}<+\infty$, so in particular we have 
	\begin{align*}
		x=\lim_{m \to +\infty}G_m^{\cO}\left(x\right).
	\end{align*}
	On the other hand, if $D\not\in \NN^{<\infty}$, an application of Lemma~\ref{lemmagaps1} to the sequence $\left(n_i\right)_{i\in D}$ gives 
	\begin{align*}
		\lim_{\substack{i\to +\infty\\ i\in D}}G_{n_i}^{\cO}\left(x\right)=x. 
	\end{align*}
\end{proof}
In the next section, we will use a   variant of the argument of Lemma~\ref{lemmaeachlog} to prove that if a VPQG basis is democratic, it is quasi-greedy. 
\section{Cesàro quasi-greedy bases, de la Valée Poussin quasi-greedy bases, and democracy. }\label{relation}

In the previous section, we have studied the convergence of the greedy algorigthm through subsequences, for VPQG bases. Some of the arguments we gave thereof will be of use in this section, where we study the connection between VPQG bases and democracy, concretely, we study the relation with the almost greediness. We will also study subsequences of VPQG bases and further conditions for equivalence between VPQG and QG bases. 

In \cite{DKKT2003}, the authors introduced the concept of almost greedy bases, which is fair to say is one of the three most studied kinds of greedy-type bases - the two others are greedy and quasi-greedy bases. 

\begin{defi}\label{definitionalmostgreedy}Let $\XB$ be a basis of a Banach space $\XX$. We say that $\XB$ is $C$-almost greedy if
\begin{align*}
&\left\Vert x-P_A\left(x\right)\right\Vert\le C \left\Vert x-P_B\left(x\right)\right\Vert&&\forall x\in \XX, m\in \NN, A\in \cG\left(x,m,1\right), B\in \NN^{\le m}.
\end{align*}
\end{defi}
A classical result of greedy approximation states that a basis is almost greedy if and only if it is quasi-greedy and democratic \cite[Theorem 3.6]{DKKT2003}. Below, we extend this result to the context of CQG and VPQG bases. 

\begin{proposition}\label{propositionvpqgdemphilog}Let $\XB$ be a de la Vallée Poussin quasi-greedy basis. If $\XB$ is democratic, it is quasi-greedy. 
\end{proposition}
\begin{proof}
	Suppose $\XB$ is $K_1$-VPQG. Since it is democratic, by Remark~\ref{remarkUCC}, it is superdemocratic, say with constant $K_2$. Given $x\in \XX$ and $\cO$ a greedy ordering for $x$, let
	\begin{align*}
		D:=&\left\lbrace m\in \NN: \osc\left(x,\left\lbrace k_{m+1}, \dots, k_{2m}\right\rbrace\right)\le 4 \right\rbrace. 
	\end{align*}
	To find a uniform upper bound for $\left(\left\Vert G_n^{\cO}\left(x\right)\right\Vert\right)_{n\in \NN}$, we consider first the following two cases:  
	\begin{enumerate}[\rm 1)]
		\item \label{inE}If $m\in D$, by Lemmas~\ref{lemmasortoftqg3} and~\ref{lemmagaps1} we have
		\begin{align}
			&\sup_{\left\vert b_j\right\vert\le 1\;\forall 1\le j\le m} \left\Vert \sum_{j=1}^{m}b_j\xx_{k_{m+j}}^*\left(x\right)\xx_{k_{m+j}}\right\Vert\le \Psi\left(1, \frac{1}{4}\right)\left\Vert x\right\Vert; \nonumber\\
			&\left\Vert G_{m+j}\left(x\right)\right\Vert\le \left(K_1+2\Psi\left(1,\frac{1}{4}\right)\right)\left\Vert x\right\Vert\; \forall 0\le j\le m. \label{agb2}
		\end{align}
		\item \label{2mnotinE2} If $m\in D$ and $2m\not\in D$, let 

%
%
\begin{align*}
			J_m:=&\left\lbrace i\in \NN: 2^lm \not\in D\quad\forall 1\le l\le i\right\rbrace.
		\end{align*}
		Note that either $J_m=\NN$ or $J_m$ is a finite interval of natural numbers, that $1\in J_m$ and that if $i\in J_m$, then
				\begin{align*}
			\left\vert \xx_{k_{2^i m+1}}^*\left(x\right)\right\vert \le& \frac{1}{4^{i-1}} \left\vert \xx_{k_{2m+1}}^*\left(x\right)\right\vert.
		\end{align*}
		From this, the triangle inequality, convexity and the superdemocracy condition we deduce that for all such $i$, 
		\begin{align*}
			&\sup_{\left\vert b_j\right\vert\le 1\;\forall 1\le j\le 2^im}\left\Vert\sum_{j=1}^{2^im}b_j\xx_{k_{2^im+j}}^*\left(x\right)\xx_{k_{2^im+j}}\right\Vert\\
			&\le \sum_{d=1}^{2^i}\sup_{\left\vert b_j\right\vert\le 1\;\forall 1+\left(d-1\right)m\le j\le dm}\left\Vert\sum_{j=1+\left(d-1\right)m}^{dm}  b_j\xx_{k_{2^im+j}}^*\left(x\right)\xx_{k_{2^im+j}} \right\Vert \\
			&\le \frac{2^{i}K_2}{4^{i-1}} \left\vert \xx_{k_{2m+1}}^*\left(x\right)\right\vert\left\Vert \Ind_{\varepsilon\left(x\right),\left\lbrace k_{m+1},\dots, k_{2m}\right\rbrace}\right\Vert \underset{\substack{ m\in D}}{\le} \frac{2K_2}{2^{i-1}}\Psi\left(1,\frac{1}{4}\right)\left\Vert x\right\Vert.
		\end{align*}
		Hence, 
		\begin{align}
			&\sum_{i\in J_m}\sup_{\left\vert b_j\right\vert\le 1\;\forall 1\le j\le 2^im}\left\Vert\sum_{j=1}^{2^im}b_j\xx_{k_{2^im+j}}^*\left(x\right)\xx_{k_{2^im+j}}\right\Vert\le 4K_2\Psi\left(1,\frac{1}{4}\right)\left\Vert x\right\Vert.\label{agb5}
		\end{align}
	\end{enumerate} 
	Now given $n\in \NN$, we can find an upper bound for $\left\Vert G_n\left(x\right)\right\Vert$ as follows. 
	\begin{enumerate}[\rm (i)]
		\item  \label{notinE1} If there is $m\in D$ such that $m\le n\le 2m$,  by \eqref{agb2} we have
		\begin{align*}
			\left\Vert G_n\left(x\right)\right\Vert\le& \left(K_1+2\Psi\left(1,\frac{1}{4}\right)\right)\left\Vert x\right\Vert.
		\end{align*}
		\item   \label{notinE12} If there is no $m\in D$ such that $m\le n\le 2m$, let 
		\begin{align*}
			m_1:=&\max\left\lbrace m\in D: m<n\right\rbrace. 
		\end{align*}
		Note that $m_1$ is well-defined because $1\in D$. Define $J_{m_1}$ as in the proof of \ref{2mnotinE2}. By hypothesis $2m_1<n$, thus $2m_1\not\in D$ and $1\in J_{m_1}$. \\
		We claim that there is $i\in J_{m_1}$ such that $2^{i}m_1< n\le 2^{i+1}m_1$. Indeed, since $2m_1<n$, if this claim were false $J_{m_1}$ would have to be a finite interval of natural numbers, say  $J_{m_1}=\left\lbrace 1,\dots, \beta\right\rbrace$ with $2^{\beta+1}m_1<n$. On the other hand, from the definition of $J_{m_1}$ we would get  $2^{\beta+1}m_1\in D$, contradicting our choice of $m_1$. \\
		From the above claim, \eqref{agb2}, \eqref{agb5} and the triangle inequality we obtain
		\begin{align*}
			\left\Vert G_n\left(x\right)\right\Vert\le& \left\Vert G_{2m_1}\left(x\right)\right\Vert+\sum_{i\in J_{m_1}}\sup_{\left\vert b_j\right\vert\le 1\;\forall 1\le j\le 2^im_1}\left\Vert\sum_{j=1}^{2^im_1}b_j\xx_{k_{2^im_1+j}}^*\left(x\right)\xx_{k_{2^im_+j}}\right\Vert\\
			\le& \left(K_1+2\Psi\left(1,\frac{1}{4}\right)+4K_2\Psi\left(1,\frac{1}{4}\right)\right) \left\Vert x\right\Vert. 
		\end{align*}
	\end{enumerate}
	Since we have considered all possible cases, we conclude that $\XB$ is $K_4$-quasi-greedy, with 
	\begin{align*}
		K_4=& K_1+2\Psi\left(1,\frac{1}{4}\right)+4K_2\Psi\left(1,\frac{1}{4}\right).
	\end{align*}
\end{proof}
Combining Proposition~\ref{propositionvpqgdemphilog} with the classical characterization of almost greedy bases of \cite[Theorem 3.3]{DKKT2003} we get the following. 
\begin{cor}\label{corollaryag} Let $\XB$ be a basis of a Banach space $\XX$. The following are equivalent: 
	\begin{itemize}
		\item $\XB$ is almost greedy. 
		\item $\XB$ is Cesàro quasi-greedy and democratic. 
		\item $\XB$ is de la Vallée Poussin quasi-greedy and democratic. 
	\end{itemize}
\end{cor}
Our next result shows another sufficient condition for the equivalence between VPQG and QG bases.
\begin{lemma}\label{lemmaSpread}Let $\XB$ be a basis of a Banach space $\XX$. If $\XB$ is VPQG and has a subsequence with a spreading model equivalent to the canonical unit vector basis of $\mathtt{c}_{0}$, it is quasi-greedy. 
\end{lemma}
\begin{proof}
	It is easy to check that the spreading model condition entails that there is $M>0$ such that, for every $m_1, m_2\in \NN$, there is $A>m_1$ with $\left\vert A\right\vert=m_2$ such that
	\begin{align}
		&\left\Vert \Ind_{\varepsilon, A}\right\Vert\le M&&\forall \varepsilon\in \EE^A.\label{smallcharac}
	\end{align}
	Suppose $\XB$ is $K$-VPQG, fix $x\in \XX$ with finite support, $m\in \NN$, and let $\cO=\left(k_j\right)_{j\in \NN}$ be a greedy ordering for $x$. Choose $A>\supp\left(x\right)$ with $\left\vert A\right\vert=m$ so that \eqref{smallcharac} holds, and let 
	\begin{align*}
		y:=&x+\left\vert \xx_{k_{m}}^*\left(x\right)\right\vert \Ind_{A}.
	\end{align*}
	Pick $\cO'=\left(k_j'\right)_{j\in \NN}$ a greedy ordering for $y$ so that 
	\begin{align*}
		&k_j=k_j'\;\forall 1\le j\le m;&& A=\left\lbrace k_j'\right\rbrace_{m+1\le j\le 2m}. 
	\end{align*}
	Since $G_m^{\cO}\left(x\right)=G_m^{\cO'}\left(y\right)$, by Remark~\ref{remarkVPQG},  Lemma~\ref{lemmasortofqg}, \eqref{smallcharac} and convexity, we have
	\begin{align*}
		\left\Vert G_{m}^{\cO}\left(x\right)\right\Vert\le& \left\Vert G_{m}^{\cO'}\left(y\right)+\sum_{j=1}^{m}\left(\frac{m+1-j}{m}\right) \xx_{k'_{m+j}}^*\left(y\right) \xx_{k'_{m+j}}\right\Vert\\
		+&\left\Vert \sum_{j=1}^{m}\left(\frac{m+1-j}{m}\right)\xx_{k'_{m+j}}^*\left(y\right) \xx_{k'_{m+j}}\right\Vert\\
		\le& K\left\Vert y\right\Vert +\left\Vert \sum_{j=1}^{m}\left(\frac{m+1-j}{m}\right)\xx_{k'_{m+j}}^*\left(y\right) \xx_{k'_{m+j}}\right\Vert\\
		\le& K\left\Vert x\right\Vert+K\left\Vert \left\vert \xx_{k_{m}}^*\left(x\right)\right\vert \Ind_{A}\right\Vert+ \left\vert \xx_{k_{m}}^*\left(x\right)\right\vert \sup_{\varepsilon\in \EE^{A}}\left\Vert \Ind_{\varepsilon, A}\right\Vert\\
		\le& \left(K+KM\alpha_2+M\alpha_2\right)\left\Vert x\right\Vert. 
	\end{align*}
	This proves the quasi-greedy condition for elements with finite support, and the general case follows by a standard density and ``small perturbations'' argument (see for example \cite[Corollary 2.3]{O2018}).
\end{proof}
Proposition~\ref{propositionvpqgdemphilog} and Lemma~\ref{lemmaSpread} can be combined by the classical result of \cite[Proposition 5.3]{DKK2003} to extract from any VPQG basis a quasi-greedy subsequence. 
\begin{cor}\label{corollarysubsec}If $\XB$ is a VPQG basis of $\XX$, then either it is quasi-greedy or it has an almost greedy subsequence. 
\end{cor}
\begin{proof}
	If $\XB$ has a subsequence with a spreding model equivalent to the unit vector basis of $\mathtt{c}_{0}$, by Lemma~\ref{lemmaSpread} it is quasi-greedy. \\
	If $\XB$ is weakly null and does not have a subsequence with a spreding model equivalent to the unit vector basis of $\mathtt{c}_{0}$, by \cite[Proposition 5.3]{DKK2003} and Proposition~\ref{propositionvpqgdemphilog}  it has an almost greedy subsequence. \\
	If $\XB$ is not weakly null, there exist $\epsilon>0$,  $x^*\in \XX^*$ and a subsequence $\YB=\left(\xx_{n_i}\right)_{i\in \NN}$ such that $\left\vert x^*\left(\xx_{n_i}\right)\right\vert\ge \epsilon$ for all $i\in \NN$. Since $\XB$ is unconditional for constant coefficients (Lemma~\ref{lemma: VPQG->NU}), it follows that there is $C>0$ such that
	\begin{align*}
		&C\left\vert A\right\vert \ge \left\Vert \Ind_{\varepsilon, A}\right\Vert \ge  \frac{\left\vert A\right\vert}{C}, &&\forall A\subset\left\lbrace n_i: i\in \NN\right\rbrace, \forall \varepsilon\in \EE^A. 
	\end{align*}
	Hence, $\YB$ is superdemocratic. Given that $\YB$ is a VPQG basis of $\overline{\langle \YB\rangle}$, by Proposition~\ref{propositionvpqgdemphilog} it is almost greedy. 
\end{proof}

\section{Open questions}
In view of the results from the previous section, the natural question that arises is whether a Cesàro quasi-greedy or VPQG condition implies the convergence of the TGA.

\textbf{Question 1.} If $\XB$ is a CQG or a VPQG basis of $\mathbb X$, is $\XB$ quasi-greedy?

It is obvious that if the basis is unconditional, then it is quasi-greedy. The converse is false, but in \cite{AA2016}, the authors proved that if the basis is 1-quasi-greedy, that is, if $\|G_n(x)\| \leq \|x\|$ for all $x\in\mathbb X$ and $n\in\mathbb N$, then the basis is unconditional. Therefore, the constant 1 in the TGA convergence property helps us recover the unconditionality of the basis.

\textbf{Question 2.} If $\XB$ is a $1$-CQG basis of $\mathbb X$, is $\XB$ $C$-quasi-greedy for some $C\geq 1$?\bigskip

\textbf{Acknowledgments:} the authors thank  F. Albiac and T. Oikhberg for their help in the preparation of this paper.

\end{document}